\pdfoutput=1

\documentclass[letterpaper, 11pt]{amsart}
\usepackage{amsmath, amsfonts, amssymb, amsthm, array, stmaryrd, graphicx, hyperref}

\newtheorem{thm}{Theorem}[section]
\newtheorem{lemma}{Lemma}[section]
\newtheorem{prop}[lemma]{Proposition}
\newtheorem{cor}[lemma]{Corollary}

\newtheorem{remark}[lemma]{Remark}

\theoremstyle{remark}
\numberwithin{equation}{section}

\newcommand{\II}{I\hspace{-0.075cm}I}
\newcommand{\III}{I\hspace{-0.075cm}I\hspace{-0.075cm}I}

\def\esmath{\ensuremath}

\def\NN{\esmath\mathbb N} 
\def\RR{\esmath\mathbb R} 
\def\CC{\esmath\mathbb C} 
 
\def\HH{\esmath\mathbb H}

\def\XX{\esmath\mathbb X}
\def\EE{\esmath\mathbb E}
\def\GG{\esmath\mathbb G}

\def\fh{\esmath\mathfrak h}

\newcommand{\p}{\partial}

\newcommand{\bn}{\begin{enumerate}}
\newcommand{\en}{\end{enumerate}}
\newcommand{\bi}{\begin{itemize}}
\newcommand{\ei}{\end{itemize}}
\newcommand{\bqq}{\begin{eqnarray*}}
\newcommand{\eqq}{\end{eqnarray*}}
\newcommand{\balg}{\begin{align*}}
\newcommand{\ealg}{\end{align*}}

\newcommand{\limt}[2]{\lim\limits_{{#1}\to{#2}}}

\newcommand{\sums}[2]{\sum\limits_{#1}^{#2}}

\newcommand{\interior}[1]{\mbox{\raisebox{0.2ex}{$\stackrel{\circ}{#1}$}}}

\newcommand{\verteq}[0]{\rotatebox{90}{$=$}}

\DeclareMathOperator{\tr}{Tr}

\DeclareMathOperator{\spec}{Spec}

\DeclareMathOperator{\rc}{Ric}
\DeclareMathOperator{\Rm}{Rm}

\begin{document}
\title[Stability of complex hyperbolic space under Ricci flow]{Stability of complex hyperbolic space under curvature-normalized Ricci flow}
\author{Haotian Wu}
\address{Department of Mathematics, The University of Texas at Austin}
\email{hwu@math.utexas.edu}
\begin{abstract}
Using the maximal regularity theory for quasilinear parabolic systems, we prove two stability results of complex hyperbolic space under the curvature-normalized Ricci flow in complex dimensions two and higher. The first result is on a closed manifold. The second result is on a complete noncompact manifold. To prove both results, we fully analyze the structure of the Lichnerowicz Laplacian on complex hyperbolic space. To prove the second result, we also define suitably weighted little H\"{o}lder spaces on a complete noncompact manifold and establish their interpolation properties.
\end{abstract}
\maketitle
\section{Introduction}
The Ricci flow, introduced by Hamilton \cite{H82}, is an important nonlinear geometric evolution equation, and it can be viewed as a dynamical system on the space of Riemannian metrics modulo diffeomorphisms. It is therefore interesting to study the \emph{dynamical} (also called geometric) stability of Ricci flow solutions. In particular, if $g_0$ is a fixed point of the Ricci flow, one may ask whether the solution $\tilde g(t)$ converges for all initial data $\tilde g_0$ that are sufficiently close to $g_0$ in some appropriate topology.

This question has been addressed for compact flat and Ricci-flat solutions of Ricci flow by Guenther, Isenberg, and Knopf \cite{GIK02}. They proved the \emph{linear} stability of Ricci-flat metrics, i.e., that the spectrum of the elliptic differential operator in the linearized equation at a Ricci-flat metric has the proper sign. Then using the maximal regularity theory developed by Da Prato and Grisvard \cite{DPG79}, they concluded the presence of a center manifold in the space of Riemannian metrics and the dynamical stability for any metric whose Ricci flow converges to a flat metric. The same authors \cite{GIK06} demonstrated the linear stability of homogeneous Ricci solitons. The convergence and stability of locally $\RR^N$-invariant solutions of Ricci flow was obtained by Knopf \cite{Kno09}. Williams \cite{MW1012} generalized Knopf's results on the volume-rescaled $\RR^N$-locally invariant solutions and the curvature-normalized Ricci flow.

In \cite{Sesum06}, \v{S}e\v{s}um strengthened the results of \cite{GIK02}. In particular, she proved that the variational stability of Ricci flow, which is defined by the nonpositivity of the second variation of Perelman's $\mathcal F$-functional, together with an integrability condition imply the dynamical stability. As a consequence, she obtained the dynamical stability for K3-surfaces. Using the spin$^{c}$ structure, Dai, Wang, and Wei \cite{DWW07} showed that K\"{a}hler-Einstein metrics with non-positive scalar curvature are stable as the critical points of the total scalar curvature functional. Combining their results with \v{S}e\v{s}um's theorem, the authors established the dynamical stability of compact K\"{a}hler-Einstein manifolds with non-positive scalar curvature \cite{DWW07} .

The stability question has also been addressed for the normalized flows. Ye \cite{Ye93} proved when the real dimension $n\geq 3$ and the metric has nonzero sectional curvature, closed Ricci-pinched solutions (defined in \cite{Ye93}) of the volume-normalized Ricci flow converge to an Einstein metric. Li and Yin \cite{LY10} obtained stability of the hyperbolic metric on $\HH^n$ under the curvature-normalized Ricci flow in dimension $n\geq 6$ assuming that the perturbation is small and decays sufficiently fast at spatial infinity. Schn{\"u}rer, Schulze, and Simon \cite{SSS11} established stability of the hyperbolic metric on $\HH^n$ under the scaled Ricci harmonic map heat flow in dimension $n\geq 4$. The proofs in these papers rely on showing that a perturbed solution converges exponentially fast to a stationary solution in the $L^2$-norm.

More recently, Bamler \cite{Bam1004} proved that every finite volume hyperbolic manifold (possibly with cusps) in dimension $n\geq 3$ is stable under the curvature-normalized Ricci flow. Later, Bamler \cite{Bam1011} obtained more general stability results for symmetric spaces of noncompact type under the curvature-normalized Ricci flow. His results make use of an improved $L^1$-decay estimate for the heat kernel in vector bundles as well as elementary geometry of negatively curved spaces.

In this paper, we study the stability of complex hyperbolic space under the curvature-normalized Ricci flow following the approach in \cite{GIK02, GIK06, Kno09}. We refer the interested reader to \cite[Section 2]{Kno09} for a detailed introduction on the maximal regularity theory. For our purpose, we apply the theory in three steps:
\bn 
\item Modify the Ricci flow so that a complex hyperbolic space is a fixed point of the modified flow.
\item Linearize the modified flow at a complex hyperbolic space, and study the spectrum of the elliptic operator in the linearized equation.
\item Set up Banach spaces of tensor fields with good interpolation properties, and apply Simonett's Stability Theorem, cf. Theorem \ref{stab}.
\en

Let $(\mathbb{CH}^m, g_B)$ denote the complete noncompact complex $m$-dimensional hyperbolic space equipped with the Bergman metric $g_B$. $g_B$ is K\"{a}hler-Einstein and has constant holomorphic sectional curvature $-c$ ($c>0$). Let $(M^{n},g_0)$ denote a smooth closed quotient of $(\mathbb{CH}^m, g_B)$, $M^{n}$ has real dimension $n=2m$. Let $\NN$ be the set of positive integers. We can now state the main theorems with respect to the spaces introduced in Section 4. 

\begin{thm}\label{thm1}
Let $m\in\NN$, $m\geq 2$, and $n=2m$. For each $\rho\in(0,1)$, there exists $\eta\in(\rho,1)$ such that the following is true for $(M^{n},g_0)$.

There exists a neighborhood $\mathcal U$ of $g_0$ in the $\fh^{1+\eta}$-topology such that for all initial data $\tilde g(0)\in\mathcal U$, the unique solution $\tilde g(t)$ of the curvature-normalized Ricci flow \eqref{eq:mRF} exists for all $t\geq 0$ and converges exponentially fast in the $\fh^{2+\rho}$-norm to $g_0$.
\end{thm}

\begin{thm}\label{thm2}
Let $m\in\NN$, $m\geq 2$, and fix $\tau>m/2$. For each $\rho\in(0,1)$, there exists $\eta\in(\rho,1)$ such that the following is true for $(\mathbb{CH}^m, g_B)$.

There exists a neighborhood $\mathcal U$ of $g_B$ in the $\mathfrak h^{1+\eta}_\tau$-topology such that for all initial data $\tilde g(0)\in\mathcal U$, the unique solution $\tilde g(t)$ of the curvature-normalized Ricci-DeTurck flow \eqref{eq:mRDF} exists for all $t\geq 0$ and converges exponentially fast in the $\mathfrak h^{2+\rho}_\tau$-norm to $g_B$. 
\end{thm}

We note that Theorem \ref{thm1} is a slight improvement of \cite[Theorem 1.6]{DWW07}, which has convergence in the $C^k$-norm ($k\geq 3$) for initial data close in the $C^k$-topology as in \cite{Sesum06}. Our method takes optimal advantage of the smoothing properties of the underlying quasilinear parabolic operator of Ricci flow in continuous interpolation spaces. The results in \cite{DWW07} and \cite{Sesum06} are more general and are obtained by different techniques.

One may also compare Theorem \ref{thm2} with the $\mathbb{CH}^m$ case in \cite[Theorem 1.2]{Bam1011}, which imposes stronger assumptions on the initial metric and weaker assumptions on the perburbations at spatial infinity, and proves convergence of the Ricci flow solutions to the fixed point in the pointed Cheeger-Gromov sense. In constrast, our initial metric is close in a less regular topology, and we impose stronger assumptions on the perturbations at spatial infinity to prove convergence in a considerably stronger sense.

A major effort of this paper is to define suitably weighted little H\"{o}lder spaces $\mathfrak h^{k+\alpha}_{\tau}$ on the complete noncompact $\mathbb{CH}^m$ and to establish the interpolation properties of these weighted spaces in all complex dimensions, cf. Theorem \ref{interp}. The interpolation results are then used in the application of the maximal regularity theory to prove the dynamical stability result, cf. Theorem \ref{thm2}.

The idea of defining weighted spaces is natural when the underlying manifold is complete noncompact and the growth rate of the volume of geodesic balls on this manifold can be estimated from above. If these weighted spaces satisfy interpolation properties, then one can prove stability results for Ricci flow on this complete noncompact manifold using the maximal regularity theory. For example, this approach works on $\RR^n$ because the interpolation theory on $\RR^n$ (with the background metric chosen to be the Euclidean metric) is well known, see for example \cite{Lun09}. The stability theorems for Ricci flow on $\RR^n$ obtained this way complement the existing results in the literature \cite{Wu93, OW07, SSS08, KL11}. One may also adapt the definition of the weighted spaces in this paper to the (noncompact) real hyperbolic space $\HH^n$ ($n\geq 3$) and establish interpolation properties accordingly. In this way one can obtain stability results that complement those in \cite{LY10, SSS11}.

The paper is organized as follows. In Section 2, we set up notations and recall the linearization formulae for the (curvature-normalized) Ricci flow. In Section 3, we study the spectrum of the elliptic operator in the linearized equation in all complex dimensions, and as a consequence, we obtain strict linear stability of the curvature-normalized Ricci-DeTurck flow at a complex $m$-dimensional hyperbolic space for $m\geq2$. In Section 4, we define the Banach spaces $\fh^{k+\alpha}$ and $\mathfrak h^{k+\alpha}_\tau$, and state the interpolation theorem for $\mathfrak h^{k+\alpha}_\tau$, cf. Theorem \ref{interp}. The proof of Theorem \ref{interp} is technical, so we postpone it to Section \ref{sect7} to improve the readability of the paper. In Sections 5 and 6, we prove Theorems \ref{thm1} and \ref{thm2} respectively. In Appendix \ref{appxA}, we include some computations. In Appendix \ref{appxB}, we include Simonett's Stability Theorem for completeness.

\subsection*{Acknowledgments} I sincerely thank my advisor Prof. Dan Knopf for his valuable advice, generous support, kind encouragement, and good humor during my learning and research. I also thank Prof. Justin Corvino for critiquing an earlier version of the manuscript, and Jiexian Li for helpful discussions.

\section{Preliminaries}
Let $(M^n,g)$ be an $n$-dimensional Riemannian manifold. $g$ is said to be negatively curved \emph{Einstein} if $\rc(g)=-\lambda g=\frac{R}{n}g$, where $\lambda>0$ is a constant. For an Einstein metric the traceless Ricci tensor $\interior\rc:=\rc-\frac{R}{n}g$ vanishes, so in this case the Riemann curvature tensor $\Rm$ admits the decomposition
\begin{align*}\Rm=\frac{R}{2n(n-1)} g\owedge g+W,\end{align*}
where $\owedge$ stands for the Kulkarni-Nomizu product, and $W$ is the Weyl curvature tensor.

When $n=2$ and $3$, $W\equiv 0$, so the Einstein condition implies $\Rm=\frac{R}{2n(n-1)} g\owedge g$, which is equivalent to $g$ having constant sectional curvature. Conversely, constant sectional curvature implies $\interior\rc$ and $W$ vanish, in all dimensions. When $n\geq 4$,  $W$ vanishes if and only if $(M,g)$ is locally conformally flat.

Let $U$ be an open set of $M^n$. When possible, we take $U=M^n$, e.g., when $M^n=\mathbb{CH}^m$. We denote by $\mathcal T^2$ the vector space of covariant two-tensor fields over $U$. We denote by $\mathcal S^2$ ($\mathcal S^2_c$, $\mathcal S^2_+$, respectively) the vector space of \emph{symmetric} covariant two-tensor fields (with compact support, positive-definite, respectively) over $U$, and let $\mathcal S_2$ be the dual of $\mathcal S^2$. In the rest of the paper, the regularity of the tensor fields will either be specified or dictated by the context. We let $\bigwedge_2$ denote the vector space of alternating two-vector fields over $U$. $\bigwedge_2$ has real dimension $\frac{n(n-1)}{2}$, and $\mathcal S^2$ has real dimension $\frac{n(n+1)}{2}$. Since $\bigwedge_2$ is isomorphic to the space of anti-symmetric covariant two-tensor fields over $U$, $\mathcal T^2\cong\mathcal S^2\oplus \bigwedge_2$. We define $\Omega^1$ to be the space of one-forms over $U$.

We denote by $\mathcal L$ the Lie derivative, by $\delta=\delta_g:\mathcal S^2\to \Omega^1$ the divergence operator (with respect to $g$), and by $\delta^\ast=\delta^\ast_g:\Omega^1\to\mathcal S^2$ the formal $L^2$-adjoint of $\delta$. $d\mu_g$ denotes the volume form of $g$, and we write $d\mu$ whenever there is no ambiguity. $\sharp:\Omega^1\to\Gamma(TM)$ is the duality isomorphism induced by $g$.  Throughout the paper, we use the Einstein summation convention.

We denote by $\langle\cdot,\cdot\rangle$ the tensor inner product with respect to $g$. Given an orthonormal frame field $\{e_i\}_{i=1}^n$ on $U$, we use the convention
\begin{align*}
\left\langle e_i\otimes e_j,e_k\otimes e_\ell\right\rangle=\delta_{ik}\delta_{j\ell}. 
\end{align*}
Define
\begin{align*}
e_i\wedge e_j := e_i\otimes e_j - e_j\otimes e_i,\quad e_ie_j:= e_i\otimes e_j+e_j\otimes e_i,
\end{align*}
then the set 
\begin{align*}\beta=\left\{\frac{1}{\sqrt 2}e_i\wedge e_j: 1\leq i<j\leq n \right\}\end{align*}
forms an orthonormal frame field of $\bigwedge_2$ over $U$, and the set
\begin{align*}
 \gamma=\left\{\frac{1}{2}e_ie_i,\; \frac{1}{\sqrt 2}e_ie_j:1\leq i<j\leq n\right\}
\end{align*}
forms an orthonormal frame field of $\mathcal S_2$ over $U$.

The Riemann curvature tensor $\Rm$ induces by its symmetries an action $R_{\bigwedge}:\bigwedge_2\to\bigwedge_2$ defined by
\begin{align*}
 \langle R_{\bigwedge}(e_i\wedge e_j), e_k\wedge e_\ell \rangle = 4R(e_i,e_j,e_\ell,e_k).
\end{align*}
$\Rm$ also induces an action $R_S:\mathcal S_2\to\mathcal S_2$ defined by 
\begin{align*}
 \langle R_{S} (e_i e_j), e_p e_q \rangle & = R(e_i,e_p,e_q,e_j)+R(e_j,e_p,e_q,e_i)\\
&\quad+R(e_i,e_q,e_p,e_j)+R(e_j,e_q,e_p,e_i).
\end{align*}
Since $\mathcal S^2$ is dual to $\mathcal S_2$, $R_S$ acts on $\mathcal S^2$ by the same action: if $h,k\in\mathcal S^2$, then
\begin{align*}
 \langle R_S(h),k\rangle = 4 R_{ipqj}h^{ij}k^{pq}.
\end{align*}
We can represent $R_{\bigwedge}$ by a matrix $R_\beta$ in the $\beta$-basis, and $R_{S}$ by a matrix $R_\gamma$ in the $\gamma$-basis. Both $R_{\bigwedge}$ and $R_S$ encode the curvature information.

We now recall some well-known facts about the Ricci flow, see for example \cite{GIK02}. The Ricci flow is a one-parameter family of Riemannian metrics $g(t)$, $0\leq t < T\leq\infty$, evolving by
\begin{align}\label{eq:RF}
\frac{\p}{\p t}g&=-2\rc(g).
\end{align}
For $h\in\mathcal S^2$ sufficiently differentiable, the linearized Ricci flow is given by
\begin{align}
\frac{\p}{\p t}h&=\Delta_L h +\mathcal L_{(\delta G(h))^\sharp} g,
\end{align}
where $G(h)=h-\frac{1}{2}(\tr_g h)g$, and $\Delta_L$ is the Lichnerowicz Laplacian. In local coordinates, we have
\begin{align}\label{eq:lichlap}
(\Delta_L h)_{ij}&=(\Delta h)_{ij}+2R_{ipqj}h^{pq}-R_i^kh_{kj}-R_j^kh_{ki},
\end{align}
where $\Delta h=g^{ij}\nabla_i\nabla_j h$ is the rough Laplacian of $h$.
The Ricci-DeTurck flow is
\begin{align}\label{eq:RDF}
\frac{\p}{\p t}g&=-2\rc(g)-P_u(g),
\end{align}
where $P_u(g):=-2\delta^\ast(\tilde u\delta(G(g,u)))$, with $u\in\mathcal S^2_+$, $G(g,u)=u-\frac{1}{2}(\tr_gu)g$, and $\tilde u:\Omega^1\to\Omega^1$ given by $(\tilde u \beta)_j := g_{jk}u^{k\ell}\beta_\ell$. The Ricci-DeTurck flow is strictly parabolic.

Since we are interested in the dynamics of Ricci flow near a complex hyperbolic space, compact or noncompact, we choose to study the curvature-normalized Ricci flow
\begin{align}\label{eq:mRF}
 \frac{\p}{\p t}g=-2(\rc(g)+\lambda g),
\end{align}
so then $(\mathbb{CH}^m,g_B)$ or $(M^n, g_0)$ becomes a fixed point of this modified flow.
$(\mathbb{CH}^m,g_B)$ or $(M^n, g_0)$ is also a fixed point of the curvature-normalized Ricci-DeTurck flow
\begin{align}\label{eq:mRDF}
\frac{\p}{\p t}g=-2(\rc(g)+\lambda g)-P_{g_B}(g).
\end{align}
The right hand side of \eqref{eq:mRDF} is a strictly elliptic operator on $g$. Alternatively, on a closed manifold, one may choose to normalize the volume of the manifold and study the volume-normalized Ricci flow as in \cite{Ye93}. 

\section{Linearization and linear stability}
\subsection{Linearization in any complex dimension} Let $m\in\NN$ and consider $(\mathbb{CH}^m, g_B)$. Recall that $g_B$ is Einstein, $\rc(g_B)=-\lambda g_B$ ($\lambda >0$), and $g_B$ has constant holomorphic sectional curvature $-c$ ($c>0$). In particular, $(\mathbb{CH}^m, g_B)$ is a Riemannian manifold of real dimension $n=2m$ with quarter-pinched sectional curvature $K\in [-c,-c/4]$. The results in this section remain valid for any smooth closed quotient $(M^n,g_0)$ of $(\mathbb{CH}^m, g_B)$.

For $h\in\mathcal S^2$, the linearization of equation \eqref{eq:mRDF} at $g_B$ is
\begin{align}
\frac{\p}{\p t} h =  \Delta_L h-2\lambda h
\end{align}
by \cite[Proposition 3.2]{GIK02}.

Define the linear operator $A:\mathcal S^2\to\mathcal S^2$ by
\begin{align}
 A h := \Delta_L h-2\lambda h.
\end{align}
Equation \eqref{eq:lichlap} implies that $A$ is the rough Laplacian plus zero-order terms, so $A$ is a strictly elliptic and self-adjoint operator on $L^2(\mathcal S_c^2)$. Since $g_B$ is Einstein, the Ricci tensor (after raising one index) acts on $(1,1)$-tensors with eigenvalues $-\lambda$. So we have
\begin{align*}(Ah)_{ij}=(\Delta h)_{ij}+2R_{ipqj}h^{pq}.\end{align*}

For now we denote by $(M,g)$ an arbitrary complete Einstein manifold. $M$ is either noncompact, e.g., $(\mathbb{CH}^m,g_B)$, or closed, e.g., $(M^n,g_0)$. We let $(\cdot,\cdot)$ be the $L^2$-pairing on $S^2_c$ defined by
\begin{align*}
 (h,k) = \int_M\langle h, k\rangle d\mu.
\end{align*}
\begin{remark}
 Recall our convention that $\left\langle e_i\otimes e_j,e_k\otimes e_\ell\right\rangle = \delta_{ik}\delta_{j\ell}$, then 
\begin{align*}
 (h,k) = \int_M \langle h^{ij}e_ie_j,k^{pq}e_pe_q\rangle d\mu = 4\int_M h^{ij} k_{ij}d\mu.
\end{align*}
\end{remark}

For $h\in\mathcal S^2_c$ that is sufficiently differentiable, we have
\begin{align*}
(Ah,h)= & \int_{M}  \left\langle \Delta h, h\right\rangle d\mu+8\int_{M} R_{ipqj}h^{pq}h^{ij}d\mu\\
 = & \int_{M}  \left\langle \Delta h, h\right\rangle d\mu+2\int_{M} \langle R_S(h), h \rangle d\mu.
\end{align*}
Then we integrate by parts to get
\begin{align}\label{eq:AR0} 
(Ah,h) =& -\int_{M} \left\langle\nabla h,\nabla h\right\rangle d\mu + 2\int_{M} \langle R_S(h), h\rangle d\mu.
\end{align}

We let $|h|^2:=\langle h,h\rangle$, $\Vert h\Vert^2:=(h,h)$; for a function $f$, $\Vert f\Vert^2$ is its $L^2$-norm.
\begin{lemma}
On a negatively curved Einstein manifold $(M,g)$ with $\rc(g)=-\lambda g$ ($\lambda>0$), given $h\in\mathcal S^2_c$, define a covariant three-tensor by 
\begin{align*}
T_{ijk}:=\nabla_kh_{ij}-\nabla_ih_{jk}. 
\end{align*}
Then
\begin{align*}\Vert \nabla h\Vert^2=\frac{1}{2}\Vert T \Vert^2+\Vert\delta h \Vert^2+\lambda\Vert h\Vert^2+\int_{M} \langle R_S(h), h\rangle d\mu.\end{align*}
\end{lemma}
\noindent This is Koiso's Bochner formula \cite{Koi79}, and we include its proof here for completeness.
\begin{proof} $T_{ijk}:=\nabla_k h_{ij}-\nabla_ih_{jk}$, then
\begin{align*}
\Vert T\Vert^2&=\int_{M}4\left\langle \nabla_k h_{ij}-\nabla_i h_{jk},\nabla^k h^{ij}-\nabla^i h^{jk}\right\rangle d\mu\\
&=2\Vert\nabla h\Vert^2-8\int_{M}\nabla_kh_{ij}\nabla^ih^{jk}d\mu.
\end{align*}
We integrate by parts and commute the covariant derivatives to obtain
\begin{align*}
-4\int_{M}\nabla_kh_{ij}\nabla^ih^{jk}d\mu&=4\int_{M}h^i_{j}\nabla_k\nabla_ih^{jk}d\mu\\
&=4\int_{M}h^i_j(\nabla_i\nabla_kh^{jk}+R^j_{kip}h^{pk}+R^k_{kip}h^{jp})d\mu\\
&=-\Vert\delta h \Vert^2+4\int_{M} h^i_jR_{ip}h^{jp}d\mu+4\int_{M} h^i_jR^j_{kip}h^{pk}d\mu\\
&=-\Vert\delta h \Vert^2-\lambda\Vert h\Vert^2-\int_{M} \langle R_S(h), h\rangle d\mu.
\end{align*}
Rearranging the terms proves the lemma. \end{proof}

Therefore, equation \eqref{eq:AR0} becomes
\begin{align}\label{eq:AR}
(Ah, h)&=-\frac{1}{2}\Vert T \Vert^2-\Vert\delta h \Vert^2-\lambda\Vert h\Vert^2+\int_{M} \langle R_S(h), h\rangle d\mu.
\end{align}
\begin{remark}\label{realhyp}
Applying equation \eqref{eq:AR} on a closed Riemannian manifold of dimension $n\geq 3$ and constant sectional curvature $K=-1$ (after normalizing the metric), then
\begin{align*}\int_{M} \langle R_S(h), h\rangle d\mu=4\int_{M} -(g_{ij}g_{pq}-g_{iq}g_{pj})h^{pq}h^{ij}d\mu=
-\Vert\tr_g h\Vert^2+\Vert h\Vert^2.\end{align*}
Since in this case $\lambda = n-1$, then for $h\not\equiv 0$, equation \eqref{eq:AR} implies
\begin{align*}(Ah,h)\leq -(n-2)\Vert h\Vert^2.\end{align*}
Thus, the curvature-normalized Ricci flow is strictly linearly stable at a closed negatively curved space form. This was proved in \cite[Appendix A]{KY09}.
\end{remark}
\begin{remark}\label{teich}
In Remark \ref{realhyp}, the dimension assumption $n\geq 3$ is crucial for linear stability. On a real two-dimensional (complex one-dimensional) surface of genus $\gamma>1$, however, the operator $A$ is \emph{not} strictly linearly stable. $A$ has a nullspace of complex dimension $3\gamma-3$, which is isomorphic to the space of holomorphic quadratic differentials, hence with the cotangent space to the Teichm\"{u}ller space.
\end{remark}

\subsection{Strict linear stability in complex dimensions two and higher}\label{subsect3.2}
Let $m\in\NN$. Let $U$ be the single geodesic normal coordinate chart covering $\mathbb{CH}^m$ (or one of a finite atlas of coordinate charts covering $M^n$). We fix an orthonormal frame field $\{e_i,e_{i+1}: i=2k-1, k=1,2,\ldots,m\}$ over $U$ such that the complex structure $J$ acts on this frame field by
\begin{align*}
 J: \{e_i,e_{i+1}\}\mapsto \{e_{i+1},-e_i\},\text{ for each } i=2k-1,\; k=1,2,\ldots,m.
\end{align*}
We abuse the notation and define the action $J$ on the indices by
\begin{align*}
 J(s)=t, \text{ if } J(e_s) = e_t \text{ or }J(e_s) = -e_t.
\end{align*}
For example, $J(1)=2$, $J(2)=1$.

We now define a canonical frame field $\gamma$ for $\mathcal S_2$ of real dimension $m(2m+1)$ in three groups, denoted by $\gamma_I$, $\gamma_{\II}$, and $\gamma_{\III}$, respectively. We define 
\begin{align*}
 \gamma_I^i & : =  \frac{1}{2} e_ie_i,\quad i=1,2,\ldots,2m;\\
 \gamma_{\II}^j & : = \frac{1}{\sqrt{2}} e_{2j-1} e_{2j},\quad j=1,2,\ldots,m.\\
\end{align*}
We define $2m(m-1)$ basis elements of $\gamma_{\III}$ by
\begin{align*}
\left\{\begin{array}{ccc}
 \gamma^p_{\III} &:=&\frac{1}{\sqrt 2} e_s e_t,\\
 \gamma^{p+1}_{\III} &:=&\frac{1}{\sqrt 2} e_{J(s)} e_{J(t)},
\end{array}\right.
 \text{ if }  1\leq s<t\leq m,\text{ and } J(s)\neq t,
\end{align*}
which for convenience we index by $p=1,3,5,\ldots,2m(m-1)-1$. For example, $\gamma_{\III}^1=\frac{1}{\sqrt2}e_1e_3,\gamma_{\III}^2=\frac{1}{\sqrt2}e_2e_4$; $\gamma_{\III}^3=\frac{1}{\sqrt2}e_1e_4,\gamma_{\III}^4=\frac{1}{\sqrt2}e_2e_3$; $\gamma_{\III}^5=\frac{1}{\sqrt2}e_1e_5,\gamma_{\III}^6=\frac{1}{\sqrt2}e_2e_6$, etc.

We define three matrices $A_m$, $B_m$, and $C_m$. $A_m$ is a $2m\times 2m$ matrix given by
\begin{align*}
A_m = \left(
\begin{array}{ccccccccc}
D & E & \cdots & \cdots &  E   \\
E & D & E & \cdots & \vdots   \\
\vdots & E & \ddots &\cdots &\vdots \\
\vdots & \vdots & \vdots & \ddots&E\\
E  &\cdots& \cdots & E & D
\end{array}
\right), 
\end{align*}
where
\begin{align*}
 D = \left(\begin{array}{cc}
            0 & 4 \\
	    4 & 0 \\
           \end{array}
 \right), \quad 
E = \left(\begin{array}{cc}
            1 & 1 \\
	    1 & 1 \\
           \end{array}
 \right).
\end{align*}
$B_m$ is the $m\times m$ diagonal matrix whose diagonal entries are $-4$. $C_m$ is a $2m(m-1)\times 2m(m-1)$ matrix and it is block diagonal given by
\begin{align*}
C_m = \left(
\begin{array}{ccccc}
  F  &     &   & \\
     &  F  &   &  \\
     &     & \ddots &\\
     &     &       & F
\end{array}
\right), 
\end{align*}
where
\begin{align*}
 F = \left(\begin{array}{cccc}
            -1 & 3  & &\\
	     3 & -1 & &\\
	       &    & -1 & -3\\
	       &    & -3 & -1
           \end{array}
 \right).
\end{align*}
Then we have the following lemma.
\begin{lemma}\label{R-gamma}
For each $m\in\NN$, $R_S$ is represented in the canonical $\gamma$-basis by the block diagonal matrix $R_\gamma$ with
\begin{align*}
 R_{\gamma}=-\frac{c}{4}\left(
\begin{array}{ccc}
A_m &      &    \\
    & B_m  &     \\
    &      & C_m
\end{array}
\right).
\end{align*}
\end{lemma}
\begin{proof}
 See Appendix \ref{appxA}.
\end{proof}

\begin{lemma}\label{eigenvalue}
For each $m\in\NN$, the largest eigenvalue of $R_\gamma$ is $c$.
\end{lemma}
\begin{proof}
$R_\gamma$ is a block diagonal matrix, so its eigenvalues are $-\frac{c}{4}$ times the eigenvalues of $A_m$, $B_m$, and $C_m$ \cite{HJ85}. The eigenvalues of $B_m$ are $-4$. The matrix $C_m$ is block diagonal, so its eigenvalues are given by the eigenvalues of $F$, which are $\{2,2,-4,-4\}$. It remains to understand the eigenvalues of $A_m$.

We write vectors in $\RR^{2m}$ as column vectors. Define $X\in\RR^{2m}$ by
\begin{align*}
X^T = \; (\underbrace{1,1,\cdots,1,1}_{2m\text{ of them}}).
\end{align*}
Define $Y_i\in\RR^{2m}$, $i=1,2,\ldots,m-1$, by
\begin{align*}
Y^T_{i} =\; (-1,-1,\underbrace{0,\cdots,0}_{\text{all }0},\underbrace{1}_{(2i+1)\text{-th entry}},1,\underbrace{0,\cdots,0}_{\text{all }0}).
\end{align*}
Define $Z_i\in\RR^{2m}$, $i=1,2,\ldots,m$, by
\begin{align*}
Z^T_{i} =\; (\underbrace{0,\cdots,0}_{\text{all }0},\underbrace{-1}_{(2i-1)\text{-th entry}},1,\underbrace{0,\cdots,0}_{\text{all }0}).
\end{align*}
Then by a direct computation, we have
\begin{align*}
A_m X & =  2(m+1)X;\\
A_m Y_i & =  2 Y_i,\quad i=1,2,\ldots, m-1;\\
A_m Z_i & =  -4 Z_i,\quad i=1,2,\ldots, m.
\end{align*}
Thus, the eigenvalues of $A_m$ are
\begin{align*}
 \{2(m+1),\underbrace{2,\cdots,2}_{(m-1)\text{ of them }},\underbrace{-4,\cdots,-4}_{m\text{ of them }} \}.
\end{align*}

The lemma follows after multiplying the eigenvalues of $A_m$, $B_m$, and $C_m$ by $-\frac{c}{4}$.\end{proof}

Therefore, we have the following estimate.
\begin{prop}\label{spec}
Let $m\in\NN$. For all $h\in\mathcal S^2_c \setminus\{0\}$,
\begin{align}\label{eq:ineq}
(Ah,h)\leq -\frac{m-1}{2}c\Vert h\Vert^2, \quad c>0.
\end{align}
In particular, the curvature-normalized Ricci-DeTurck flow \eqref{eq:mRDF} is strictly linearly stable at $(\mathbb{CH}^m,g_B)$ or $(M^n,g_0)$ when $m\geq 2$.
\end{prop}
\begin{proof}
The sum of the entries of $-\frac{c}{4}A_m$ is the scalar curvature $R$ of $g_B$, so $R=-m(m+1)c$, and hence
\begin{align*}
 \lambda = -\frac{R}{2m} =\frac{m+1}{2}c.
\end{align*}

Let $\lambda_S$ denote the largest eigenvalue of $R_\gamma$, then $\lambda_S=c$ by Lemma \ref{eigenvalue}. Recall the variational characterization
\begin{align*}
 \lambda_S = \sup\limits_{h\in\mathcal S^2_c\setminus\{0\}} \frac{\langle R_S h,h\rangle}{\langle h,h\rangle}.
\end{align*}
Then equation \eqref{eq:AR} implies
\begin{align*}
(Ah,h) & \leq -\lambda\Vert h\Vert^2 + \int \langle R_S(h), h\rangle \\
& \leq -\frac{m+1}{2}c\Vert h\Vert^2+c\Vert h\Vert^2 \\
& \leq -\frac{m-1}{2}c\Vert h\Vert^2,\quad c>0.
\end{align*}
The asserted strict linear stability follows if $m\geq2$. \end{proof}

\begin{remark}
On a closed manifold $(M^n, g_0)$, we can assume $h\in\mathcal S^2$. On the complete noncompact $(\mathbb{CH}^m,g_B)$, we will relax the assumption $h\in\mathcal S^2_c$, cf. Corollary \ref{eqholds2}.
\end{remark}

Inequality \eqref{eq:ineq} is sharp since when $m=1$, the operator $A$ has a nontrivial nullspace, cf. Remark \ref{teich}. On the other hand, if on $M^n$ there is a smooth one-parameter family of complex hyperbolic metrics with some fixed complex structure, then $A$ would have a nontrivial null eigenspace. Proposition \ref{spec} shows that the operator $A$ has trivial null eigenspace when $m\geq 2$. Hence, we recover the following well-known result: the moduli space of complex hyperbolic metrics on a closed 4-manifold is locally rigid in the sense that if a complex hyperbolic metric $g$ is near $g_0$ in some norm, then $g$ differs from $g_0$ by a homothetic scaling and a diffeomorphism on $M^4$ \cite{IK05, Kim02, LeB95}. Moreover, we have extended this local rigidity to any closed $2m$-manifold when $m\geq 2$.

\section{(Weighted) little H\"{o}lder spaces}
On a closed Riemannian manifold $M^n$ admitting a complex hyperbolic metric, we fix a background metric and a finite atlas $\{ U_\upsilon \}_{1\leq\upsilon\leq\Upsilon}$ of coordinate charts covering $M^n$. Given a smooth $h\in\mathcal S^2$ over $M^n$, for each integer $k\geq 0$ and real number $\alpha\in(0,1)$, we denote 
\begin{align*}
 [h_{ij}]_{k+\alpha;U_\upsilon} := \sup\limits_{|\ell|=k} \sup\limits_{x\neq y \in U_\upsilon} \frac{\left|\nabla^\ell h_{ij}(x)- \nabla^\ell h_{ij}(y)\right|}{d(x,y)^\alpha},
\end{align*}
where $\ell$ is a multi-index and $\nabla^\ell h_{ij}= \nabla^{\ell_1}_1\nabla^{\ell_2}_2\cdots\nabla^{\ell_n}_n h_{ij}$.

We define the $(k+\alpha)$-H\"{o}lder norm of $h$ by
\begin{align*}
\Vert h \Vert_{k+\alpha} := 
\sup_{\substack{1\leq i,j\leq n,\\ 1\leq\upsilon\leq\Upsilon}}\left(\sum\limits_{m=0}^{k}\sup\limits_{|\ell|=m} \sup\limits_{x\in U_\upsilon}\left|\nabla^\ell h_{ij}(x)\right|+[h_{ij}]_{k+\alpha;U_\upsilon}\right).
\end{align*}
The components $\{h_{ij}\}$ are with respect to the fixed atlas $\{ U_\upsilon \}_{1\leq\upsilon\leq\Upsilon}$, whereas $\nabla$, $|\cdot |$, and the distance function $d$ are computed with respect to a fixed background metric. Different finite atlases or background metrics yield equivalent H\"{o}lder norms.

The little H\"{o}lder space $\fh^{k+\alpha}$ on the closed manifold $M^n$ is defined to be the completion of $C^\infty$ symmetric covariant two-tensor fields in $\Vert \cdot \Vert_{k+\alpha}$. The space $\fh^{k+\alpha}$ is separable \cite{GIK02}. In particular, $h\in \fh^{k+\alpha}$ has the following property \cite[Proposition 0.2.1]{Lun95}:
\begin{align*}
 \limt{t}{0+} \sup\limits_{|\ell|=k} \sup_{\substack{x\neq y\in U_\upsilon,\\ d(x,y)\leq t}}\frac{\left|\nabla^\ell h_{ij}(x) - \nabla^\ell h_{ij}(y)\right|}{d(x,y)^\alpha} = 0
\end{align*}
for all $1\leq \upsilon\leq \Upsilon$ and $1\leq i,j\leq n$. This property also defines $\fh^{k+\alpha}$ \cite{Lun95}.

On a complete noncompact manifold the definition of the H\"{o}lder norm depends on the choice of the atlas and the background metric. On $\mathbb{CH}^m$, $m\in\NN$, we will use the single geodesic normal coordinate chart given by the exponential map at the origin and choose the background metric to be $g_B$. In this case, $|\cdot|$, $\Vert\cdot\Vert$, and $d$ are computed with respect to $g_B$.

Let $h\in\mathcal S^2$, $h$ is not necessarily compactly supported over $\mathbb{CH}^m$. To apply equation \eqref{eq:AR} and Proposition \ref{spec}, $h$ should decay at spatial infinity with the decay rate dictated by the geometry under consideration. To apply Simonett's Stability Theorem to deduce dynamical stability from linear stability, $h$ should also belong to a Banach space that satisfies certain interpolation properties. These considerations lead us to define the following suitably weighted little-H\"{o}lder spaces.

We first set up some notations. Let $B_R$ be a geodesic ball of radius $R$ centered at the origin of our single chart. Fix $m\in\NN$, consider the open covering of $\mathbb{CH}^m$ by the family $\{A_N\}_{N\in\NN}$ of overlapping annuli and a disk defined by $A_1:=B_4$, and $A_N:=B_{N+3}\setminus B_{N-1}$ for $N\geq 2$. If $x,y\in A_N$, denote $d^N_x:=d(x,\p A_N)$, $d^N_{x,y}:=\min\{d_x,d_y\}$, and we write $d_x, d_{x,y}$ whenever there is no ambiguity. Given a smooth $h\in\mathcal S^2$ over $\mathbb{CH}^m$, for integers $k,q\geq 0$, multi-index $\ell$, and $\alpha\in(0,1)$, we let
\begin{align*}
 |h_{ij}|'_{q; A_N} & := \sup\limits_{|\ell|=q} \sup_{x\in A_N} d_x^q |\p^\ell h_{ij}(x)|,\\ 
 [h_{ij}]'_{k+\alpha; A_N} & := \sup\limits_{|\ell|=k} \sup_{x\neq y\in A_N} d_{x,y}^{k+\alpha} \frac{|\p^\ell h_{ij}(x)- \p^\ell h_{ij}(y)|}{d(x,y)^\alpha}.
\end{align*}

For each $m\in\NN$, we fix $\tau=(m+\xi)/2$ where $\xi>0$ and define the $\tau$-weighted $(k+\alpha)$-H\"{o}lder norm of $h$ by
\begin{align*}
\Vert h \Vert_{k+\alpha;\tau} := 
\sup_{1\leq i,j\leq n} \sup_{N\in\NN}e^{N\tau}\left(\sum\limits_{q=0}^{k}|h_{ij}|'_{q; A_N}
+[h_{ij}]'_{k+\alpha; A_N}\right).
\end{align*} 
The $\tau$-weighted little H\"{o}lder space $\mathfrak h^{k+\alpha}_\tau$ on the complete noncompact manifold $\mathbb{CH}^m$ is defined to be the closure of $C_c^\infty$ symmetric covariant two-tensor fields in $\Vert\cdot \Vert_{k+\alpha;\tau}$.

Similar to the little H\"{o}lder space $\mathfrak h^{k+\alpha}$, we have the following infinitesimal property for $h\in \fh^{k+\alpha}_\tau$.
\begin{lemma}
If $h\in\fh^{k+\alpha}_\tau$, define
\begin{align*}
 F_{h}(t) := \sup\limits_{1\leq i,j\leq n}\sup\limits_{N\in\NN}e^{N\tau}\left(\sup\limits_{|\ell|=k} \sup_{\substack{x\neq y\in A_N,\\ d(x,y)\leq t}} d_{x,y}^{k+\alpha}\frac{|\p^\ell h_{ij}(x)- \p^\ell h_{ij}(y)|}{d(x,y)^\alpha}\right),
\end{align*}
then
\begin{align}\label{eq:lholder}
 \limt{t}{0+} F_h(t) = 0.
\end{align}
\end{lemma}
\begin{proof}
Given $h\in\fh^{k+\alpha}_\tau$, there exists a sequence of $C_c^\infty$ symmetric two-tensor fields $\{h_n\}_{n\in\NN}$ that converge to $h$ in $\fh^{k+\alpha}_\tau$. Note that each $h_n$ satisfies property \ref{eq:lholder}. Then we estimate
\begin{align*}
 F_h(t)\leq F_{(h-h_n)}(t) + F_{h_n}(t).
\end{align*}
In particular, 
\begin{align*}
 F_{(h-h_n)}(t) \leq \Vert h-h_n \Vert_{k+\alpha;\tau}.
\end{align*}
For $\varepsilon>0$, choose $n$ large such that $\Vert h-h_n \Vert_{k+\alpha;\tau}<\varepsilon/2$, and choose $\delta >0$ such that $F_{h_n}(t)<\varepsilon/2$ if $t<\delta$, then $F_h(t)<\varepsilon$ if $t<\delta$.
\end{proof}
\begin{remark}
From now on, we shorten some expressions. For integer $k\geq0$, $\nabla^k$ (or $\p^k$) means first applying $\nabla^\ell$ (or $\p^\ell$) and then taking the supremum over $|\ell|=k$. We omit the indices of $h$, and the result will follow after either summing or taking the supremum over $1\leq i,j\leq n=2m$.
\end{remark}

We denote by $W_1^2(\mathbb{CH}^m)$ the Sobolev space of symmetric two-tensor fields $h$ over $\mathbb{CH}^m$ such that
\begin{align*}
 \sums{i,j=1}{2m} \int_{\mathbb{CH}^m} \left( |h_{ij}|^2+|\nabla h_{ij}|^2 \right) d\mu_{g_B} < \infty,
\end{align*}
where $\nabla$ is computed with respect to the background metric $g_B$. We denote by $\hookrightarrow$ a continuous embedding.

\begin{lemma}\label{dom}
Let $k,m\in\NN$, $\alpha\in(0,1)$, and fix $\tau>m/2$. Then $\mathfrak h^{k+\alpha}_\tau \hookrightarrow W_1^2(\mathbb{CH}^m)$.
\end{lemma}
\begin{proof}
Recall that $B_R$ denotes a geodesic ball of radius $R$ centered at the origin in $\mathbb{CH}^m$. The volume of $B_R$, denoted by $V(B_R)$, has exponential growth rate $m$, i.e.,  there is a positive constant $\tilde C=\tilde C(m)$ such that $V(B_R)\leq \tilde C e^{mR}$, see for example \cite{Gold99}. We know $\nabla h = \p h + \Gamma\ast h$, where $\ast$ denotes contraction using $g_B$, and $\Gamma$ is uniformly bounded by some constant $C$ in geodesic normal coordinates on the homogeneous space $(\mathbb{CH}^m, g_B)$. In what follows and the rest of the paper, ``$A\lesssim B$'' means $A\leq C B$, where $C>0$ is a constant that may change from line to line.

Let $h\in \mathfrak h^{k+\alpha}_\tau$. Recall that $\tau=(m+\xi)/2$ where $\xi>0$. If $x\in B_2$, then $x\in A_1$ with $d_x\geq 1$, so we have
\begin{align*}
\int_{B_2} \left(|h(x)|^2 + |\nabla h(x)|^2 \right)d\mu_{g_B}(x) & \lesssim \int_{B_2} \left( |h(x)|^2+|\p h(x)|^2 \right) d\mu_{g_B}(x)\\
& \leq \int_{B_2} \left[ (|h(x)|'_{0;A_1})^2+d_x^{-2}(|h(x)|'_{1;A_1})^2 \right] d\mu_{g_B}(x)\\
& \leq \int_{B_2} \left[ (|h(x)|'_{0;A_1})^2+(|h(x)|'_{1;A_1})^2 \right] d\mu_{g_B}(x)\\
& \leq \Vert h\Vert^2_{k+\alpha;\tau}e^{-2\tau}\int_{B_2} d\mu_{g_B}\\
& \lesssim \Vert h\Vert^2_{k+\alpha;\tau}e^{-2\tau} e^{2m}\\
& = \Vert h\Vert^2_{k+\alpha;\tau}e^{-\xi}e^{-m}e^{2m}\\
& = e^{m-\xi} \Vert h\Vert^2_{k+\alpha;\tau}.
\end{align*}

Similarly, if $x\in B_{2N}\setminus B_{2N-2}$, $N\geq 2$, then $x\in A_{2N-2}=B_{2N+1}\setminus B_{2N-3}$ with $d_x\geq 1$. So $|h(x)|\leq |h|'_{0;A_{2N-2}}$, $|\p h(x)|\leq d_x^{-1}|h|'_{1;A_{2N-2}}\leq |h|'_{1;A_{2N-2}}$, and we estimate
\begin{align*}
\int_{B_{2N}\setminus B_{2N-2}} \left(|h(x)|^2 + |\nabla h(x)|^2 \right) d\mu_{g_B}(x) & \lesssim e^{-2(2N-2)\tau} \Vert h\Vert^2_{k+\alpha;\tau}\int_{B_{2N}\setminus B_{2N-2}}d\mu_{g_B}(x)\\
& \lesssim e^{-(2N-2)\xi}\Vert h\Vert^2_{k+\alpha;\tau} e^{-(2N-2)m}e^{2Nm}\\
& = e^{2m-(2N-2)\xi}\Vert h\Vert^2_{k+\alpha;\tau}.
\end{align*} 

Then, with $B_0=\emptyset$, we have
\begin{align*}
\int_{B_{2N}}\left( |h(x)|^2+|\nabla h(x)|^2 \right) d\mu_{g_B}(x) &= \sums{K=1}{N}\int_{B_{2K}\setminus B_{2K-2}} \left( |h(x)|^2+|\nabla h(x)|^2 \right) d\mu_{g_B}(x)\\
& \lesssim \Vert h\Vert^2_{k+\alpha;\tau} \left( e^{m-\xi} + \sums{K=2}{N} e^{2m-(2K-2)\xi}\right)\\
& \leq \Vert h\Vert^2_{k+\alpha;\tau} \left(e^{m-\xi} + e^{2m} \frac{e^{-2\xi}}{1-e^{-2\xi}}\right).
\end{align*}
Let $N\to\infty$, we see $h\in W_1^2(\mathbb{CH}^m)$.
\end{proof}

\begin{lemma}\label{eqholds}
Let $k, m\in\NN$, $\alpha\in(0,1)$, and fix $\tau>m/2$. Then for all $h\in \mathfrak h^{k+\alpha}_\tau$, equation \eqref{eq:AR} holds.
\end{lemma}
\begin{proof} $(\mathbb{CH}^m, g_B)$ is a smooth complete Riemannian manifold with infinite injectivity radius. The Ricci curvature of $g_B$ and all its covariant derivatives are bounded. So $C_c^\infty$ tensor fields are dense in $W_1^2(\mathbb{CH}^m)$ \cite{Aub82, Heb99}. By Lemma \ref{dom}, $\mathfrak h^{k+\alpha}_\tau\hookrightarrow W_1^2(\mathbb{CH}^m)$ for any $k, m\in\NN$ and $\alpha\in(0,1)$. So equation \eqref{eq:AR} holds for all $h\in \mathfrak h^{k+\alpha}_\tau$ by strong convergence in $W^2_1$-norm. \end{proof}

Thus, we extend Proposition \ref{spec} to the following.
\begin{cor}\label{eqholds2}
Let $k,m\in\NN$, $m\geq 2$, $\alpha\in(0,1)$, and fix $\tau>m/2$. For all $h\in \mathfrak h^{k+\alpha}_\tau\setminus\{0\}$, the curvature-normalized Ricci-DeTurck flow \eqref{eq:mRDF} is strictly linearly stable at $(\mathbb{CH}^m,g_B)$ in the sense of Proposition \ref{spec}.
\end{cor}

Consider two Banach spaces $X, Y$ with $Y\overset{d}\hookrightarrow X$, where $\overset{d}\hookrightarrow$ denotes a continuous and dense embedding. Given a real number $\theta\in(0,1)$, one can define the \emph{continuous} interpolation space $(X, Y)_\theta$. The interpolation methods are well known in the literature, see for example \cite{Am95, Lun09, Sim95}, and we will recall their precise definitions in Section \ref{sect7}. We now state two interpolation results for the weighted little H\"{o}lder spaces $\mathfrak h^{k+\alpha}_\tau$. We postpone their proofs to Section \ref{sect7}. Analogous results hold for the little H\"{o}lder spaces $\mathfrak h^{k+\alpha}$ \cite{Trie78} and have been used in \cite{GIK02, KY09, Kno09}.
\begin{lemma}\label{a7}
Let $m\in\NN$ and fix $\tau>m/2$. Let $0\leq k \leq \ell$ be integers, $0<\alpha\leq\beta<1$, and $0<\theta<1$, then there exists a constant $C(\theta)>0$ depending on $\theta$ such that for all $h\in \mathfrak h^{\ell+\beta}_\tau$,
\begin{align}\label{eq:a7-1}
 \Vert h \Vert_{(\mathfrak h^{k+\alpha}_\tau,\; \mathfrak h^{\ell+\beta}_\tau)_\theta} \leq C(\theta) \Vert h \Vert^{1-\theta}_{\mathfrak h^{k+\alpha}_\tau}\Vert h \Vert^{\theta}_{\mathfrak h^{\ell+\beta}_\tau}.
\end{align}
\end{lemma}

\begin{thm}\label{interp}
Under the assumptions of Lemma \ref{a7}, if $(1-\theta)(k+\alpha)+\theta(\ell+\beta)\notin\NN$, then there is a Banach space isomorphism
\begin{align*}
(\mathfrak h^{k+\alpha}_{\tau},\mathfrak h^{\ell+\beta}_{\tau})_{\theta}\cong \mathfrak h^{(1-\theta)(k+\alpha)+\theta(\ell+\beta)}_{\tau},
\end{align*}
with equivalence of the respective norms.
\end{thm}
Consequently, $\mathfrak h^{\ell+\beta}_\tau\overset{d}\hookrightarrow (\mathfrak h^{k+\alpha}_\tau,\mathfrak h^{\ell+\beta}_\tau)_\theta \overset{d}\hookrightarrow \mathfrak h^{k+\alpha}_\tau$ for $\ell\geq k$, $\beta\geq\alpha$.

\section{Dynamical stability of $(M^n, g_0)$}
For fixed $0<\sigma<\rho<1$, consider the following densely and continuously embedded spaces:
\begin{align*}\begin{array}{ccccccc}
\mathbb X_1 & \overset{d}\hookrightarrow & \mathbb E_1 & \overset{d}\hookrightarrow &\mathbb X_0 & \overset{d}\hookrightarrow & \mathbb 
E_0\\
\verteq	& & \verteq & & \verteq & & \verteq\\
\fh^{2+\rho} & & \fh^{2+\sigma} & &\fh^{0+\rho}& & \fh^{0+\sigma}\\
\end{array}.\end{align*}
For fixed $1/2\leq\beta<\alpha<1$, define
\begin{align*}\mathbb X_\beta:=(\mathbb X_0,\mathbb X_1)_\beta,\quad\mathbb X_\alpha:=(\mathbb X_0,\mathbb X_1)_\alpha.\end{align*}
For fixed $0<\epsilon\ll 1$, define
\begin{align*}\mathbb \GG_\beta:=\{g\in\mathbb X_\beta: g>\epsilon g_0\},\quad \mathbb \GG_\alpha:=\mathbb \GG_\beta\cap\mathbb X_\alpha,\end{align*}
where ``$g>\epsilon g_0$'' means $g(X,X)>\epsilon g_0(X,X)$ for any tangent vector $X$.

We abbreviate the right hand side of equation \eqref{eq:mRDF} by $Q_{g_0}(g)g$. $Q_{g_0}(g)$ is a quasilinear elliptic operator. By a straightforward computation as in \cite[Lemma 3.1]{GIK02}, we have the following lemma.
\begin{lemma}\label{coef}
If we express $Q_{g_0}(g)g$ in terms of the first and second derivatives of $g$ in local coordinates, then
\begin{align}\label{eq:coef}
(Q_{g_0}(g)g)_{ij}& =a(x,g_0,g_0^{-1},g,g^{-1})_{ij}^{k\ell pq}\p_p\p_qg_{k\ell}\\
&\quad + b(x,g_0,g_0^{-1},\p g_0,g,g^{-1},\p g)_{ij}^{k\ell p}\p_p g_{k\ell} \notag \\
&\quad\quad +c(x,g_0^{-1},\p g_0,\p^2 g_0,g)_{ij}^{k\ell}g_{k\ell}. \notag
\end{align}
The coefficients $a, b, c$ depend smoothly on $x\in M^n$, and they are analytic functions of their remaining arguments.
\end{lemma}
\begin{remark}
This lemma holds true for the operator $Q_{g_B}(g)$.
\end{remark}

We point out two typos in \cite[Lemma 3.1]{GIK02}: the coefficient $b$ missed the dependence on $\p g$, and the coefficient $c$ missed the dependence on $\p^2 g_0$. 

We are now ready for the proof of Theorem \ref{thm1}.
\begin{proof}[Proof of Theorem \ref{thm1}] Proposition \ref{spec} applied to $h\in\XX_1$ (or $h\in\EE_1$) implies the spectrum of the operator $A$, denoted by $\spec(A)$, satisfies $\spec(A)\subset (-\infty,-c/2]$ for $c>0$ provided $m\geq 2$. We then apply the Stability Theorem (Theorem \ref{stab}) to conclude Theorem \ref{thm1}. We omit the details since they are similar to the proof in \cite{GIK02}.
\end{proof}

\section{Dynamical stability of $(\mathbb{CH}^m, g_B)$}
In this section, $|\cdot|$, $\Vert\cdot\Vert$, $\nabla$, $\Gamma_{ij}^k$, and the distance function $d$ are computed with respect to the fixed background metric $g_B$, and recall $\tau>m/2$ is fixed.

For fixed $0<\sigma<\rho<1$, consider
\begin{align*}\begin{array}{ccccccc}
\mathbb X_1 & \overset{d}\hookrightarrow & \mathbb E_1 & \overset{d}\hookrightarrow &\mathbb X_0 & \overset{d}\hookrightarrow & \mathbb 
E_0\\
\verteq	& & \verteq & & \verteq & & \verteq\\
\mathfrak h^{2+\rho}_\tau & & \mathfrak h^{2+\sigma}_\tau & &\mathfrak h^{0+\rho}_\tau& & \mathfrak h^{0+\sigma}_\tau\\
\end{array}.\end{align*}
For fixed $\frac{1}{2}\leq\beta<\alpha<1-\frac{\rho}{2}$, define
\begin{align*}\mathbb X_\alpha:=(\mathbb X_0,\mathbb X_1)_\alpha,\quad \mathbb X_\beta:=(\mathbb X_0,\mathbb X_1)_\beta.
\end{align*}
By Theorem \ref{interp}, $\XX_\alpha\cong\mathfrak h^{2\alpha+\rho}_\tau$ and $\XX_\beta\cong \mathfrak h^{2\beta+\rho}_\tau$ with equivalence of the respective norms. Note that $\XX_\alpha \overset{d}\hookrightarrow \XX_\beta \overset{d}\hookrightarrow \mathfrak h^{1+\rho}_\tau$.

For a fixed $\hat g\in \XX_\beta$, \eqref{eq:coef} allows us to view $Q_{g_B}(\hat g)$ as a linear operator on $\XX_1$: if $h\in\XX_1$, then
\begin{align}\label{eq:linop}
(Q_{g_B}(\hat g) h)_{ij} & =a(x,g_B,g_B^{-1},\hat g,\hat g^{-1})_{ij}^{k\ell pq}\p_p\p_q h_{k\ell}\\
&\quad + b(x,g_B,g_B^{-1},\p g_B,\hat g,\hat g^{-1},\p\hat g)_{ij}^{k\ell p}\p_p h_{k\ell} \notag \\
&\quad\quad +c(x,g_B^{-1},\p g_B, \p^2 g_B, \hat g)_{ij}^{k\ell} h_{k\ell}. \notag
\end{align}
In fact, $Q_{g_B}(g_B)$ is $\Delta_{g_B}$ plus lower order terms with bounded coefficients, and in particular, $Q_{g_B}(g_B)g_B=0$. In the $\gamma$-basis defined in Section \ref{subsect3.2} and for fixed $i,j,k,\ell$, we can represent the coefficient $a_{ij}^{k\ell pq}$ of the second order term in $Q_{g_B}(g_B)$ as a matrix $a_\gamma$. If we let $\lambda$ and $\Lambda$ denote the smallet and the largest eigenvalues of the matrix $a_\gamma$ respectively, then $\Lambda/\lambda =1$ on any $B_R\subset\mathbb{CH}^m$. Since the coefficient $a$ depends analytically on $\hat g$, if $\hat g$ is sufficiently close to $g_B$ in $\XX_\beta$, then $Q_{g_B}(\hat g)$ satisfies $0< \hat c < \Lambda/\lambda < \hat C$ for some constants $\hat c,\hat C$ and for all $1\leq i,j,k,\ell\leq n=2m$. We call such $\hat g$ an \emph{admissible perturbation} of $g_B$.

For $0<\epsilon\ll 1$ to be chosen, cf. Lemma \ref{bddop2}, define an open set in $\XX_\beta$ by
\begin{align*}
\mathbb \GG_\beta := \{g\in\mathbb X_\beta \text{ is an admissible perturbation}: g>\epsilon g_B\},
\end{align*}
and define
\begin{align*}
\mathbb \GG_\alpha := \mathbb \GG_\beta\cap\mathbb X_\alpha.
\end{align*}
 
We denote by $ L_{\hat g}:=Q_{g_B}(\hat g)$ the unbounded linear operator on $\mathbb X_0$ with dense domain $D(L_{\hat g})=\mathbb X_1$. We extend $L_{\hat g}$ to $\hat L_{\hat g}$, which is now defined on $\mathbb E_0$ with dense domain $D(\hat L_{\hat g})=\EE_1$. If $X,Y$ are two Banach spaces, we denote by $\mathcal L(X,Y)$ the space of bounded linear operators from $X$ to $Y$.

\begin{lemma}\label{bddop1}${}$
\bn
\item $\hat g\mapsto L_{\hat g}$ is an analytic map $\GG_\beta \to \mathcal L(\XX_1,\XX_0)$.
\item $\hat g\mapsto \hat L_{\hat g}$ is an analytic map $\GG_\alpha \to \mathcal L(\EE_1,\EE_0)$.
\en 
\end{lemma}
\begin{proof} Let $\hat g\in \GG_\beta$ be given, we abbreviate equation \eqref{eq:linop} as
\begin{align*}
Q_{g_B}(\hat g)h = a(x,\hat g)\ast \p^2 h+b(x,\hat g,\p\hat g)\ast \p h+c(x,\hat g)\ast h,
\end{align*}
where $a$, $b$, $c$ are all analytic in $\hat g$, and are polynomials in $\hat g$, $\hat g^{-1}$, $g_B$, $g_B^{-1}$, and possibly in $\p g_B$, $\p^2 g_B$. $\hat g^{-1}\in \GG_\beta$ implies $\hat g^{-1}$ is controlled by $g_B^{-1}$. In what follows, $\pi$ denotes a polynomial that may change from line to line.

For $x,y\in A_N$ and $x\neq y$, we always have $d^N_{x,y}\leq 2$. There are three possibilities:
\begin{align*}
\emph{Case 1: }& d^N_{x,y}\geq 1/2,\\
\emph{Case 2: }& d^N_x,d^N_y < 1/2, \\
\emph{Case 3: }& \text{without loss of generality}, d^N_x < 1/2 \leq d^N_y.
\end{align*}

If we are in Case 2, either $x,y\in A_{N-1}$ with $d^{N-1}_{x,y}\geq 1$, or $x,y\in A_{N+1}$ with $d^{N+1}_{x,y}\geq 1$, then the estimates for Case 1 will apply to Case 2 up to a different constant. If we are in Case $3$, then there exists $z\in A_N$ with $d^N_z = 1/2$ such that either $x,z\in A_{N-1}, d^{N-1}_{x,z}\geq 1$ and $z,y\in A_{N}, d^N_{z,y}\geq 1/2$, or $x,z\in A_{N+1}, d^{N+1}_{x,z}\geq 1$ and $z,y\in A_{N}, d^N_{z,y}\geq 1/2$. So using the triangle inequality, the estimates for Case 1 will apply to Case 3 up to a different constant. Thus, it suffices to prove the lemma for Case 1.

We now check for Case 1. Consider $x\neq y\in A_N$ with $1/2\leq d^N_{x,y}\leq 2$. Writing $d^N_{x,y}$ as $d_{x,y}$ for short, we estimate
\begin{align*}
 d_{x,y}^\rho\frac{|a(x,\hat g)\ast \p^2 h(x) - a(y,\hat g)\ast \p^2 h(y)|}{d(x,y)^\rho}&\leq  d_{x,y}^\rho(F_1+F_2),
\end{align*}
where
\begin{align*}
 F_1 & = \frac{|a(x,\hat g)\ast \p^2 h(x) -a(y,\hat g)\ast \p^2 h(x)|}{d(x,y)^\rho},
\end{align*}
\begin{align*}
 F_2 & = \frac{|a(y,\hat g)\ast \p^2 h(x) -a(y,\hat g)\ast \p^2 h(y)|}{d(x,y)^\rho}.
\end{align*}
Since $1/2\leq d_{x,y}\leq 2$ and $\XX_\beta \overset{d}\hookrightarrow \mathfrak h^{1+\rho}_\tau \overset{d}\hookrightarrow \mathfrak h^{0+\rho}_\tau $, we have
\begin{align*}
 d_{x,y}^\rho F_1 & \leq d_{x,y}^\rho\left( d_{x,y}^{-\rho}[a(\hat g)]'_{0+\rho; A_N}d_x^{-2}|h|'_{2;A_N}\right)\\
& \lesssim e^{-2N\tau}\Vert \pi(\hat g)\Vert_{0+\rho;\tau} \Vert h \Vert_{2+\rho;\tau}\\
& \lesssim e^{-2N\tau} \pi(\Vert\hat g\Vert_{\XX_\beta}) \Vert h \Vert_{2+\rho;\tau},
\end{align*}
and
\begin{align*}
d_{x,y}^\rho F_2 & \leq d_{x,y}^\rho\left(|a(\hat g)|'_{0;A_N} d_{x,y}^{-2-\rho} [h]'_{2+\rho;A_N}\right)\\
& \lesssim  e^{-2N\tau}\Vert \pi(\hat g)\Vert_{0+\rho;\tau} \Vert h \Vert_{2+\rho;\tau}\\
& \lesssim  e^{-2N\tau} \pi(\Vert\hat g\Vert_{\XX_\beta}) \Vert h \Vert_{2+\rho;\tau}.
\end{align*}
So we have
\begin{align*}
\sup_{N\in\NN} e^{N\tau }[h]'_{0+\rho;A_N} & \lesssim \pi(\Vert \hat g\Vert_{\XX_\beta})\Vert h\Vert_{2+\rho;\tau}.
\end{align*}
Likewise, we have
\begin{align*}
\sup_{N\in\NN} e^{N\tau }|h|'_{0;A_N} & \lesssim \pi(\Vert \hat g\Vert_{\XX_\beta})\Vert h\Vert_{2+\rho;\tau}.
\end{align*}
The above estimates imply
\begin{align*}
\Vert a(x,\hat g)\ast \p^2 h \Vert_{0+\rho;\tau}&\lesssim \pi(\Vert \hat g\Vert_{\XX_\beta}) \Vert h\Vert_{2+\rho;\tau}.
\end{align*} 
Similarly, we have
\begin{align*}
\Vert b(x,\hat g,\p\hat g)\ast \p h \Vert_{0+\rho;\tau}&\lesssim \Vert b(x,\hat g,\p\hat g)\Vert_{0+\rho;\tau}\Vert h\Vert_{1+\rho;\tau}\\
&\lesssim \pi(\Vert \hat g \Vert_{0+\rho;\tau}+\Vert \hat g \Vert_{1+\rho;\tau})\Vert h\Vert_{2+\rho;\tau}\\
&\lesssim \pi(\Vert \hat g\Vert_{\XX_\beta})\Vert h\Vert_{2+\rho;\tau},\\
\Vert c(x,\hat g)\ast h \Vert_{0+\rho;\tau}& \lesssim \pi(\Vert \hat g\Vert_{\XX_\beta})\Vert h\Vert_{2+\rho;\tau}.
\end{align*} 
Putting everything together, we see
\begin{align*}
\Vert L_{\hat g} h\Vert_{\XX_0}=\Vert Q_{g_B}(\hat g) h\Vert_{\XX_0}&\lesssim \Vert h\Vert_{\XX_1}.
\end{align*}
Thus, $\hat g\mapsto L_{\hat g}$ is an analytic map from $\GG_\beta$ to $\mathcal L(\XX_1, \XX_0)$, which proves part (1). Part (2) is proved analogously.
\end{proof}

Given a Banach space $X$, a linear operator $A:D(A)\subset X\to X$ is called \emph{sectorial} if there are constants $\omega\in\RR$, $\beta\in(\pi/2, \pi)$, and $M>0$ such that 
\bn
\item The resolvent set $\rho(A)$ contains 
\begin{align*}
S_{\beta,\omega}:=\{\lambda\in\CC:\lambda\neq\omega,|\arg(\lambda-\omega)|<\beta\}, 
\end{align*}
\item $\Vert (\lambda I-A)^{-1}\Vert_{\mathcal L(X,X)}\leq \frac{M}{|\lambda-\omega|}, \text{ for all }\lambda\in S_{\beta,\omega}$.
\en

\begin{lemma}\label{bddop2}
$\hat L_{g_B}$ is sectorial, and there exists $\epsilon >0$ in the definition of $\GG_\beta$ such that for each $\hat g \in \GG_\alpha$, $\hat L_{\hat g}$ generates a strongly continuous analytic semigroup on $\mathcal L(\EE_0,\EE_0)$.
\end{lemma}
\begin{proof}
Since $\hat L_{g_B}$ is $\Delta_{g_B}$ plus lower order terms, it is a strongly elliptic operator. The spectrum of $\Delta_{g_B}$ on $\EE_1$ is $(\infty,0]$, so that $\spec(L_{g_B})\subset (-\infty,\lambda_0)$ for some $\lambda_0\in\RR$, since the lower order terms may affect the spectrum by a bounded amount. Then by an argument similar to the proof of \cite[Lemma 3.4]{GIK02}, $\hat L_{g_B}$ is sectorial by the Schauder estimates for $\hat L_{g_B}$ with respect to the $\mathfrak h^{k+\rho}_\tau$-norms. To see such estimates hold, one first uses the standard Schauder estimates for $\hat L_{g_B}$ on $A_N$ that
\begin{align*}
 \sum\limits_{\ell=0}^{2}|h|'_{\ell; A_N}+[h]'_{2+\rho; A_N} \leq C\left ( |\hat L_{g_B} h|'_{0; A_N}+[\hat L_{g_B} h]'_{\rho; A_N} \right),
\end{align*}
where $C$ is a constant depending on the complex dimension $m$, the H\"{o}lder exponent $\rho$, and the ratio $\Lambda/\lambda$, but not on $N$ \cite{GT01}. Multiplying both sides of this inequality by $e^{N\tau}$ and taking the supremum over $N\in\NN$, then
\begin{align*}
 \Vert h \Vert_{2+\rho;\tau}\lesssim \Vert \hat L_{g_B} h\Vert_{0+\rho;\tau}.
\end{align*}
So $\hat L_{g_B}$ is sectorial, and since $\hat L_{g_B}$ is densely defined by construction, it generates a strongly continuous analytic semigroup by a standard characterization \cite[pp. 34]{Lun95}. By part (2) of Lemma \ref{bddop1}, we can choose $\epsilon >0$ in the definition of $\GG_\beta$ so small that for $\hat g\in \GG_\alpha$, we have
\begin{align*}
 \Vert \hat L_{\hat g}-\hat L_{g_B} \Vert_{\mathcal L(\EE_1,\EE_0)} < \frac{1}{M+1},
\end{align*}
for the constant $M>0$ in the definition of sectorial operator corresponding to $\hat L_{g_B}$. So the perturbation $\hat L_{\hat g}$ is sectorial by \cite[Proposition 2.4.2]{Lun95}, and hence generates a strongly continuous analytic semigroup on $\mathcal L(\EE_0,\EE_0)$. \end{proof}

We are now ready to prove Theorem \ref{thm2}.
\begin{proof}[Proof of Theorem \ref{thm2}] We verify conditions (B1)--(B7) in Theorem \ref{stab}.

(B1): By construction, $\XX_1\overset{d}\hookrightarrow \XX_0$ and $\EE_1\overset{d}\hookrightarrow \EE_0$ are continuous dense inclusions of Banach spaces. For fixed $\frac{1}{2}\leq\beta<\alpha<1-\frac{\rho}{2}$, $\XX_\alpha, \XX_\beta$ are continuous interpolation spaces corresponding to the inclusion $\XX_1\overset{d}\hookrightarrow \XX_0$.

(B2): Equation \eqref{eq:mRDF} is an autonomous quasilinear parabolic equation. There exists a positive integer $K$ such that for all $\hat g$ in the open set $\GG_\beta\subset \XX_\beta$, the domain $D(L_{\hat g})$ of $L_{\hat g}$ is equal to $\XX_1$, and the map $\hat g\to L_{\hat g}|_{\XX_1}$ belongs to $C^K(\GG_\beta,\mathcal L(\XX_1,\XX_0))$ by part (1) of Lemma \ref{bddop1}. In fact, we can let $K=\infty$.

(B4): By Lemma \ref{bddop2}, for each $\hat g\in \GG_\alpha$, $\hat L_{\hat g}|_{\EE_1}\in\mathcal L(\EE_1,\EE_0)$ generates a strongly continuous analytic semigroup on $\mathcal L(\EE_0,\EE_0)$.

(B3), (B5): By construction, for each $\hat g\in \GG_\beta$, $\hat L_{\hat g}$ is an extension of $L_{\hat g}$ to a domain $D(\hat L_{\hat g})$ that is equal to $\EE_0$. For each $\hat g\in \GG_\alpha, L_{\hat g}$ is the part of $\hat L_{\hat g}$ in $\XX_0$.

(B6): Recall that $0<\sigma<\rho<1$ were fixed. For each $\hat g\in \GG_\alpha$, there exists $\theta=\frac{\rho-\sigma}{2}\in(0,1)$ such that, by Theorem \ref{interp},
\begin{align*}
 (\EE_0, D(\hat L_{\hat g}))_\theta=(\mathfrak h^{0+\sigma}_\tau,\mathfrak h^{2+\sigma}_\tau)_\theta\cong \mathfrak h^{0+\rho}_\tau= X_0
\end{align*}
with equivalence of the respective norms.

Define the set $(\EE_0, D(\hat L_{\hat g}))_{1+\theta}:=\{g\in D(\hat L_{\hat g}):\hat L_{\hat g}(g)\in (\EE_0, D(\hat L_{\hat g}))_\theta\}$, and endow it with the graph norm of $\hat L_{\hat g}$ with respect to the space $(\EE_0,D(\hat L_{\hat g}))_\theta$, which we just showed to be equivalent to the space $\XX_0$. This graph norm is $\Vert \cdot\Vert_{\XX_0}+\Vert \hat L_{\hat g}(\cdot)\Vert_{\XX_0}$.

By definition, $\XX_1 \overset{d}\hookrightarrow (\EE_0, D(\hat L_{\hat g}))_{1+\theta}$. We claim the respective norms are equivalent. Indeed, $\Vert \cdot\Vert_{\XX_0}+\Vert \hat L_{\hat g}(\cdot)\Vert_{\XX_0}\lesssim \Vert \cdot\Vert_{\XX_1}$ since $\XX_1 \overset{d}\hookrightarrow \XX_0$. The opposite inequality, $\Vert \cdot\Vert_{\XX_1}\lesssim \Vert \cdot\Vert_{\XX_0}+\Vert \hat L_{\hat g}(\cdot)\Vert_{\XX_0}$, follows from Schauder estimates for $\hat L_{\hat g}$ with respect to the weighted H\"{o}lder norms $\Vert \cdot\Vert_{k+\rho;\tau}$, similar to those used in the proof of Lemma \ref{bddop2}. So the claim is true. Therefore, there is a Banach space isomorphism $(\EE_0, D(\hat L_{\hat g}))_{1+\theta}\cong\XX_1$ with equivalence of the respective norms.

(B7): This just follows from Lemma \ref{a7}.

To finish the proof, for fixed $\rho\in(0,1)$ and $\alpha\in (\frac{1}{2},1-\frac{\rho}{2})$, there exists $\eta\in(\rho,1)$ such that, by Theorem \ref{interp},
\begin{align*}
 \XX_\alpha \cong \fh^{2\alpha+\rho}_{\tau} =: h^{1+\eta}_\tau
\end{align*}
with equivalence of the respective norms. By Corollary \ref{eqholds2}, $\spec(A)\subset (-\infty,-c/2]$ for $c>0$ in complex dimension $m\geq 2$. Therefore, we can apply the Stability Theorem (Theorem \ref{stab}) to conclude Theorem \ref{thm2}.
\end{proof}

\section{Interpolation properties}\label{sect7}
Let $X, Y$ be two Banach spaces with $Y\overset{d}\hookrightarrow X$. We review the $K$-method in interpolation theory, cf. \cite{Am95, Lun09, Trie78}. 

For every $h\in X+Y$ and $t>0$, set 
\begin{align*}
 K(t,h)=K(t,h,X,Y) := \inf_{\substack{h=a+b,\\ a\in X,\,b\in Y}}\left(\Vert a\Vert_X+t\Vert b\Vert_Y \right).
\end{align*}
Let $\theta\in(0,1)$, $p\in[1,\infty]$. We define the \emph{real} interpolation space between $X$ and $Y$ by
\begin{align*}
 (X,Y)_{\theta,p} := \{h\in X+Y:t\mapsto t^{-\theta} K(t,h)\in L^p_\ast(0,\infty)\},
\end{align*}
where $L^p_\ast(0,\infty)$ is the $L^p$ space with respect to the measure $dt/t$. We abbreviate $L^p_\ast(0,\infty)$ as $L^p_\ast$. In particular, $\Vert \cdot\Vert_{L^\infty_\ast} = \Vert\cdot\Vert_{L^\infty}$. The norm of $h\in (X,Y)_{\theta,p}$ is given by
\begin{align*}
 \Vert h\Vert_{\theta,p}=\Vert h\Vert_{(X,Y)_{\theta,p}}:=\Vert t^{-\theta} K(t,h)\Vert_{L^p_\ast}.
\end{align*}
We define the \emph{continuous} interpolation space between $X$ and $Y$ by
\begin{align*}
 (X,Y)_{\theta}:=\{h\in X+Y:\limt{t}{0+} t^{-\theta} K(t,h)=0\}.
\end{align*}
Since for every $h\in X+Y$, $t\mapsto K(t, h)$ is concave and hence continuous in $(0,\infty)$, $(X,Y)_{\theta}$ is a closed subspace of $(X,Y)_{\theta,\infty}$, and it is endowed with the $(X,Y)_{\theta,\infty}$-norm, i.e., $\Vert\cdot\Vert_\theta=\Vert\cdot\Vert_{\theta,\infty}$.
\begin{remark}\label{b1}
Since $Y\overset{d}\hookrightarrow X$, we can equivalently define $(X,Y)_\theta$ to be the closure of $Y$ in the $(X,Y)_{\theta,\infty}$-norm \cite{Sim95}.
\end{remark}

We now prove Lemma \ref{a7}.
\begin{proof}[Proof of Lemma \ref{a7}] We use the general fact \cite[Corollary 1.7]{Lun09} that given two Banach spaces $X,Y$ with $Y\overset{d}\hookrightarrow X$, if $\theta\in(0,1)$ and $p\in[1,\infty]$, then there exists a constant $C(\theta, p)>0$ such that for all $y\in Y$,
\begin{align*}
 \Vert y\Vert_{\theta,p} \leq C(\theta,p) \Vert y \Vert^\theta_X \Vert y\Vert^{1-\theta}_Y.
\end{align*}
Taking $Y:=\mathfrak h^{\ell+\beta}\overset{d}\hookrightarrow\mathfrak h^{k+\alpha} =:X$, and noting when $p=\infty$, $\Vert\cdot\Vert_\theta = \Vert\cdot\Vert_{\theta,\infty}$, we obtain \eqref{eq:a7-1}.
\end{proof}

In the rest of this section, we prove Theorem \ref{interp}.
\begin{lemma}\label{interpo1}
Let $m\in\NN$ and fix $\tau>m/2$, then for each $\theta\in(0,1)$, we have the Banach space isomorphism
\begin{align*}
 (\mathfrak h^{0}_{\tau}, \mathfrak h^{1}_{\tau})_{\theta}\cong\mathfrak h^{\theta}_{\tau}
\end{align*}
with equivalence of the respective norms.
\end{lemma}
\begin{proof} We let $X:=\mathfrak h^{0}_{\tau}$, $Y:=\mathfrak h^{1}_{\tau}$, and $Z:=\mathfrak h^{\theta}_{\tau}$. Note that $Y\overset{d}\hookrightarrow Z \overset{d}\hookrightarrow X$.

We first show $(X,Y)_\theta\subset Z$, i.e., $(\mathfrak h^{0}_{\tau}, \mathfrak h^{1}_{\tau})_{\theta}\subset \mathfrak h^{\theta}_{\tau}$.

Let $h\in (X,Y)_\theta\subset X$. If $b\in Y\subset X$, then $a:=h-b\in X$. So we can always decompose $h=a+b$ with $a\in X$ and $b\in Y$, and
\begin{align*}
\Vert h\Vert_X & \leq \Vert a\Vert_X + \Vert b\Vert_X \lesssim \Vert a\Vert_X + \Vert b\Vert_Y,
\end{align*}
since $Y\overset{d}\hookrightarrow X$. Taking the infimum over all such decompositions, then
\begin{align*} 
\Vert h\Vert_X & \leq K(1,h)\leq \Vert h\Vert_\theta,
\end{align*}
since for $h\in (X,Y)_{\theta}$ we have the useful fact that for all $t>0$, $K(t,h)\leq t^\theta\Vert h \Vert_{\theta}$.

If $x\neq y\in A_N$ for some $N$, then we can find a curve $\gamma$ in $A_N$ connecting $x$ and $y$ such that $d_z\geq d_{x,y}$ for all $z\in \gamma$ and $\text{Length}(\gamma)\lesssim d(x,y)$. We have, in the single geodesic coordinate chart covering $\mathbb{CH}^m$,
\begin{align*}
|a(x)-a(y)| & \leq |a(x)|+|a(y)| \lesssim e^{-N\tau}\Vert a\Vert_X,\\
|b(x)-b(y)| &\leq \int_\gamma \left|\langle \text{grad } b,\dot{\gamma} \rangle\right|.
\end{align*}
Given a decomposition $h=a+b$ with $a\in X$, $b\in Y$, we see
\begin{align*}
 d_{x,y}^\theta|h(x)-h(y)|  & \leq d_{x,y}^\theta|a(x)-a(y)| + d_{x,y}^\theta|b(x)-b(y)|\\
& \lesssim d_{x,y}^\theta e^{-N \tau } \Vert a\Vert_X + \int_\gamma \frac{d_{x,y}^\theta}{d_z^\theta} d_z^\theta\left|\text{grad }b(z)\right|\\
&\lesssim e^{-N\tau} \Vert a\Vert_X + e^{-N\tau} \Vert b\Vert_Y d(x,y)\\
& \lesssim e^{-N\tau} K(d(x,y),h)\\
& \lesssim  e^{-N\tau} d(x,y)^\theta\Vert h\Vert_\theta.
\end{align*}
So $[h]'_{0+\theta;A_N}\leq e^{-N\tau} \Vert h\Vert_\theta$, and it follows that
\begin{align*}
\Vert h\Vert_Z&=\Vert h\Vert_X+\sup_{N\in\NN}e^{N\tau} [h]'_{0+\theta;A_N} \\
&\lesssim \Vert h \Vert_\theta + \Vert h\Vert_\theta.
\end{align*}
Therefore, $h\in Z$ with $\Vert h\Vert_Z\lesssim \Vert h\Vert_\theta$.

We now show $(X,Y)_\theta\supset Z$, i.e., $(\mathfrak h^{0}_{\tau}, \mathfrak h^{1}_{\tau})_{\theta}\supset \mathfrak h^{\theta}_{\tau}$.

Let $h\in Z\subset X$. We will suitably decompose $h=a_t+b_t$ for $t\in(0,\infty)$ with $a_t\in X$, $b_t\in Y$. When $t\in[ 1, \infty)$, let $a_t = h$, $b_t=0$, then $K(t,h)\leq \Vert h\Vert_X$, and hence
\begin{align*}
 t^{-\theta}K(t,h)\leq t^{-\theta}\Vert h\Vert_X \leq \Vert h\Vert_X\lesssim \Vert h\Vert_Z
\end{align*}
because $Z\overset{d}\hookrightarrow X$.

We now consider the case $t\in(0,1]$.

For each fixed $x\in\mathbb{CH}^m$, let $\zeta(y)=\zeta\left(d(x,y)\right)$ be a smooth positive bump function compactly supported in $\{y\in\mathbb{CH}^m : d(x,y)\leq 1\}$ with unit mass:
\begin{align*}
 1 = \int_{\mathbb{CH}^m} \zeta\left(d(x,y)\right) d\mu_{g_B}(y).
\end{align*}
Let $B_t$ denote a geodesic ball of radius $t$ in $\mathbb{CH}^m$ and $V(B_t)$ its volume. Recall that $g_B$ has constant scalar curvature $R<0$, and that the asymptotic expansion of $V(B_t)$ with respect to $t$ is
\begin{align*}
 \frac {V(B_t)}{\omega_{2m}t^{2m}} = 1 - \frac{R}{6(2m+2)}t^2 + O(t^3),
\end{align*}
where $\omega_{2m}$ is the volume of the unit ball in Euclidean $\RR^{2m}$. Then under the homothetic scaling $d(x,y)\mapsto d(x,y)/t$, there exist constants $\bar c, \bar C>0$ such that for $t \in (0,1]$,
\begin{align*}
 \bar c \leq C_t:=\frac{1}{t^n}\int_{\{d(x,y)\leq t\}} \zeta\left(\frac{d(x,y)}{t}\right) d\mu_{g_B}(y) \leq \bar C.
\end{align*}

Define 
\begin{align*}
b_t(x)&:= \frac{1}{C_t}\left[\frac{1}{t^n}\int_{\{d(x,y)\leq t\}} h(y) \zeta\left(\frac{d(x,y)}{t}\right)d\mu_{g_B}(y)\right],\\
a_t(x) &:= h(x)-b_t(x).
\end{align*}
For fixed $x\in\mathbb{CH}^m$ and $y\in\{d(x,y)\leq t\leq 1\}$, we can find some $A_N$ such that $B_1(x)\subset A_N$ with $d_{x,y}\geq 1/2$, so then $|h(x)-h(y)|/d(x,y)^\theta \leq e^{-N\tau}d_{x,y}^{-\theta} \Vert h\Vert_Z \lesssim e^{-N\tau} \Vert h\Vert_Z$.

We now estimate
\begin{align*}
|a_t(x)| & = |h(x)-b_t(x)| = \left| \frac{C_t}{C_t}h(x) - b_t(x)\right |\\ 
& \leq  \frac{1}{C_t}\frac{1}{t^n}\int_{\{{d(x,y)\leq t}\}} |h(x)-h(y)| \left|\zeta\left(\frac{d(x,y)}{t}\right)\right| d\mu_{g_B}(y) \\
& =  \frac{1}{C_t}\frac{1}{t^n}\int_{\{0<{d(x,y)\leq t}\}} \frac{|h(x)-h(y)|}{d(x,y)^\theta} d(x,y)^\theta \left|\zeta\left(\frac{d(x,y)}{t}\right)\right| d\mu_{g_B}(y) \\
& \lesssim \frac{1}{t^n}\int_{\{{d(x,y)\leq t}\}} e^{-N\tau} \Vert h \Vert_Z d(x,y)^{\theta} \left|\zeta\left(\frac{d(x,y)}{t}\right)\right| d\mu_{g_B}(y) \\
& = e^{-N\tau}\Vert h \Vert_Z \int_{\{{d(x,y)\leq t}\}}  \frac{1}{t^n} d(x,y)^{\theta} \left|\zeta\left(\frac{d(x,y)}{t}\right)\right| d\mu_{g_B}(y). 
\end{align*}
Recall the definition $\Vert a_t\Vert_X = \sup\limits_{N\in\NN} e^{N\tau}|a_t(x)|'_{0;A_N}$ and using $d(x,y)/t\leq 1$, then
\begin{align}\label{eq:estc1}
 \Vert a_t\Vert_X \lesssim \Vert h \Vert_Z \int_{\{{d(x,y)\leq t}\}} \frac{t^\theta}{t^n} \left|\zeta\left(\frac{d(x,y)}{t}\right)\right| d\mu_{g_B}(y).
\end{align}
Analogously, we have
\begin{align}\label{eq:estc2}
  \sup_{N\in\NN} e^{N\tau} |b_t(x)|'_{1;A_R} \lesssim \Vert h \Vert_Z \int_{\{d(x,y)\leq t\}} \frac{t^{\theta-1}}{t^n} \left|\p\zeta\left(\frac{d(x,y)}{t}\right)\right| d\mu_{g_B}(y).
\end{align}
Since $Z\overset{d}\hookrightarrow X$, $\theta\in(0,1)$ is fixed, and $t\in(0,1]$, we have
\begin{align}\label{eq:estc3}
 t^{1-\theta} \Vert b_t\Vert_X & =  t^{1-\theta} \Vert h-a_t\Vert_X \notag\\
& \leq t^{1-\theta}\left(\Vert h\Vert_X + \Vert a_t\Vert_X\right) \notag\\
& \lesssim \Vert h \Vert_Z \left(1+ t\int_{\{d(x,y)\leq t\}} \left|\zeta\left(\frac{d(x,y)}{t}\right)\right| d\mu_{g_B}(y)\right).
\end{align}
Putting together estimates \eqref{eq:estc1}, \eqref{eq:estc2}, and \eqref{eq:estc3}, for $t\in (0,1]$,
\begin{align*}
t^{-\theta}K(t,h) & \leq  t^{-\theta} \left(\Vert a_t\Vert_X + t\Vert b_t\Vert_Y\right)\\
& = t^{-\theta}\left(\Vert a_t\Vert_X + t\Vert b_t\Vert_X+ t \sup_{N\in\NN} e^{N\tau}| b_t(x)|'_{1;A_N}  \right)\\
& \lesssim \Vert h \Vert_Z \\ 
&\quad +\Vert h\Vert_Z \int_{\{d(x,y)\leq t\}} \left( \left|\zeta\left(\frac{d(x,y)}{t}\right)\right| + \left|\p\zeta\left(\frac{d(x,y)}{t}\right)\right|\right)  d\mu_{g_B}(y)\\
& \lesssim \Vert h\Vert_Z.
\end{align*}
Thus, we obtain for $t\in(0,1]$,
\begin{align*}
 t^{-\theta}K(t,h) & \lesssim \Vert h\Vert_Z.
\end{align*}

To conclude $h\in (X,Y)_\theta$, it remains to show that $\limt{t}{0+} t^{-\theta}K(t,h)=0$.

Fix $x\in\mathbb{CH}^m$, let $t\in(0,1]$ and $y\in\{d(x,y)\leq t\}$, then $x,y\in A_N$ for some $N\in\NN$ with $d_{x,y}\geq{1/2}$. Noting $d(x,y)/t\leq 1$, we estimate
\begin{align*}
t^{-\theta} |a_t(x)| & \leq t^{-\theta} \frac{1}{C_t}\frac{1}{t^n}\int_{\{{d(x,y)\leq t}\}} |h(x)-h(y)| \left|\zeta\left(\frac{d(x,y)}{t}\right)\right| d\mu_{g_B}(y) \\
& = \frac{1}{C_t}\frac{1}{t^n}\int_{\{0<{d(x,y)\leq t}\}} \frac{d_{x,y}^{\theta}}{d_{x,y}^{\theta}} \frac{|h(x)-h(y)|}{d(x,y)^\theta} \left(\frac{d(x,y)}{t}\right)^\theta \left|\zeta\left(\frac{d(x,y)}{t}\right)\right| d\mu_{g_B}(y) \\
& \lesssim \frac{1}{C_t}\frac{1}{t^n}\int_{\{0<{d(x,y)\leq t}\}}  d_{x,y}^{\theta}\frac{|h(x)-h(y)|}{d(x,y)^\theta}\left|\zeta\left(\frac{d(x,y)}{t}\right)\right| d\mu_{g_B}(y)\\
& \leq \sup\limits_{\substack{x\neq y\in A_N, \\ d(x,y)\leq t}}d_{x,y}^{\theta}\frac{|h(x)-h(y)|}{d(x,y)^\theta} \left(\frac{1}{C_t}\int_{\{{d(x,y)\leq t}\}} \frac{1}{t^n}\left|\zeta\left(\frac{d(x,y)}{t}\right)\right|d\mu_{g_B}(y)\right)\\
& = \sup\limits_{\substack{x\neq y\in A_N, \\ d(x,y)\leq t}}d_{x,y}^{\theta}\frac{|h(x)-h(y)|}{d(x,y)^\theta}.
\end{align*}
As $t\to 0+$, by property \eqref{eq:lholder} that for $h\in Z:=\fh^{\theta}_\tau$,
\begin{align*}
 \limt{t}{0+} \sup\limits_{N\in\NN}e^{N\tau}\sup\limits_{\substack{x\neq y\in A_N,\\d(x,y)\leq t}}d^\theta_{x,y}\frac{|h(x)- h(y)|}{d(x,y)^\theta} = 0,
\end{align*}
hence
\begin{align*}
\limt{t}{0+} t^{-\theta}\Vert a_t\Vert_X = 0.
\end{align*}
Similarly, $\limt{t}{0+} t^{1-\theta}\Vert b_t\Vert_Y = 0$. Thus, $\limt{t}{0+} t^{-\theta}K(t,h)=0$.

Therefore, $h\in (X,Y)_\theta$ with $\Vert h\Vert_\theta\lesssim \Vert h\Vert_Z$. 
\end{proof}

Let $X$ be a Banach space. Consider a linear operator $A:D(A)\to X$ such that
\begin{align}\label{eq:resolv}
(0,\infty)&\subset\rho(A),\text{ and there exists a constant } C>0 \text{ such that}\\
&\quad \Vert \lambda R(\lambda, A)\Vert_{\mathcal L(X,X)}\leq C,\text{ for all } \lambda>0, \notag
\end{align}
where $\rho(A)$ is the resolvent set of $A$ and $R(\lambda, A)=(\lambda I - A)^{-1}$ is the resolvent operator of $A$. Note this is condition (3.1) in \cite{Lun09}.
\begin{lemma}\label{C2}
Let $m\in\NN$ and fix $\tau>m/2$. For $1\leq i \leq n=2m$, consider $A_i:=\p_i$ in geodesic normal coordinates with respect to $g_B$. Then for each $i$, $A_i:=D(A_i):\mathfrak h^1_\tau \to \mathfrak h^0_\tau $ satisfies condition \eqref{eq:resolv}.
\end{lemma}
\begin{proof}
Let $h\in \mathfrak h_\tau^0$. Write $x=(x^1,\ldots,x^n) \in \mathbb{CH}^m$ in the single geodesic normal coordinates, so $0\leq x^i<\infty$, $1\leq i\leq n$. Since $(\mathbb{CH}^m,g_B)$ is a complete manifold, there exists for each $1\leq i\leq n$ a unique geodesic $\gamma_i:\RR\to\mathbb{CH}^m$ with $\gamma_i(x^i)=x$ and $\dot\gamma(x^i) = \p/\p x^i|_x$. Then for each $\lambda>0$,
\begin{align*}
(R(\lambda, A_i) h)(x) := \int_{x^i}^\infty e^{-\lambda(t-x^i)} h(\gamma_i(t)) dt
\end{align*}
is the resolvent operator of $A_i$, $1\leq i\leq n$. In particular, $(0,\infty)\subset\rho(A)$.

If $i=1$, then $M-1\leq x^1\leq M$ since $x\in B_M\setminus B_{M-1}$ for some $M\in\NN$, which corresponds to $x\in A_N$ for some $N=N(M)\in\NN$. We then estimate
\begin{align*}
 e^{N\tau}|(\lambda R(\lambda, A_1)h)(x)| & \leq  e^{N\tau}\sums{K=N}{\infty}\int_{K-1}^{K+3} \lambda e^{-\lambda(t-x^1)} |h(\gamma_i (t))| dt\\
& \leq \sums{K=N}{\infty}\int_{K-1}^{K+3} \lambda e^{-\lambda(t-x^1)}e^{N\tau}e^{-K\tau} \Vert h\Vert_{0;\tau} dt\\
& \leq \Vert h\Vert_{0;\tau} \sums{K=N}{\infty}\int_{K-1}^{K+3} \lambda e^{-\lambda(t-x^1)} dt\\
& \leq 4\Vert h\Vert_{0;\tau} \int_{M-1}^{\infty} \lambda e^{-\lambda(t-x^1)} dt\\
& \leq 4\Vert h\Vert_{0;\tau} e^{-\lambda(M-1-x^1)}\\
& \leq C_1\Vert h\Vert_{0;\tau}.
\end{align*}
Similarly, for $2\leq i\leq n$,
\begin{align*}
e^{N\tau}|(\lambda R(\lambda, A_i)h)(x)| \leq C_i \Vert h\Vert_{0;\tau}.
\end{align*}
Taking the supremum over $N\in\NN$, then for each $\lambda>0$,
\begin{align*}
 \Vert \lambda R(\lambda, A_i)h\Vert_{0;\tau} \leq \max\limits_{1\leq i\leq n}C_i \Vert h\Vert_{0;\tau}.
\end{align*}

Therefore, $A_i$ satisfies condition \eqref{eq:resolv} for $1\leq i\leq n$.
\end{proof}

Lemma \ref{C2} ensures the interpolation theory in \cite{Lun09} applies, so from the basic step Lemma \ref{interpo1} we can conclude the following.
\begin{lemma}\label{interpo2}
Let $m\in\NN$ and fix $\tau>m/2$, then for $\theta\in(0,1)$ and $\ell\in\NN$ such that $\ell\theta \notin\NN$, we have the Banach space isomorphism
\begin{align*}
 (\mathfrak h^{0}_{\tau}, \mathfrak h^{\ell}_{\tau})_{\theta}\cong\mathfrak h^{\ell\theta}_{\tau}
\end{align*}
with equivalence of the respective norms.
\end{lemma}

We can now prove Theorem \ref{interp}.
\begin{proof}[Proof of Theorem \ref{interp}]
Given $\mathfrak h^{k+\alpha}_{\tau}$, $\mathfrak h^{\ell+\beta}_\tau$, there exist $\theta_0,\theta_1\in(0,1)$ and $p \in\NN$ such that by Lemma \ref{interpo2},
\begin{align*}
(\mathfrak h^{0}_{\tau}, \mathfrak h^{p}_{\tau})_{\theta_0}\cong\mathfrak h^{p\theta_0}_{\tau}=\mathfrak h^{k+\alpha}_{\tau},\\
(\mathfrak h^{0}_{\tau}, \mathfrak h^{p}_{\tau})_{\theta_1}\cong\mathfrak h^{p\theta_1}_{\tau}=\mathfrak h^{\ell+\beta}_{\tau},
\end{align*}
with equivalence of the respective norms.

Suppose $(1-\theta)(k+\alpha)+\theta(\ell+\beta) \notin\NN$, then
\begin{align*}
 (\mathfrak h^{k+\alpha}_{\tau},\mathfrak h^{\ell+\beta}_{\tau})_\theta & \cong 
\left((\mathfrak h^{0}_{\tau}, \mathfrak h^{p}_{\tau})_{\theta_0}, (\mathfrak h^{0}_{\tau}, \mathfrak h^{p}_{\tau})_{\theta_1}\right)_\theta\\
& \cong (\mathfrak h^0_\tau, \mathfrak h^p_\tau)_{(1-\theta)\theta_0+\theta\theta_1} \quad\text{(by the Reiteration Theorem \cite{Trie78, Lun09})}\\
& \cong \mathfrak h^{(1-\theta)p\theta_0+\theta p\theta_1}_\tau \quad\text{(by Lemma \ref{interpo2})}\\
& = \mathfrak h^{(1-\theta)(k+\alpha)+\theta(\ell+\beta)}_\tau.
\end{align*}
with equivalence of the respective norms.
\end{proof}

\appendix
\section{Proof of Lemma \ref{R-gamma}}\label{appxA}
\begin{lemma}\label{entries}
 The non-zero components of the Riemann curvature tensor of $g_B$ are given by the following:
\begin{align*}
 \text{ if }& J(i)=j, \text{ then } R(e_i,e_j,e_j,e_i) =  -c; \\
 \text{ if }& J(i)\neq j, \text{ then } R(e_i,e_j,e_j,e_i) = -c/4 ;\\
\text{ if }& k<\ell, p<q,  J(k)=\ell,J(p)=q, \text{ then } \left\{\begin{array}{ccc}
 R(e_k,e_\ell,e_q,e_p) &=& -c/2, \\
 R(e_k,e_p,e_q,e_\ell) &=& -c/4,\\
 R(e_k,e_q,e_p,e_\ell) &=& c/4. 
\end{array} \right.
\end{align*}
\end{lemma}
\begin{proof}
On a K\"{a}hler manifold $(M,g)$ of constant holomorphic sectional curvature $-c$ ($c>0$), e.g., $(\mathbb{CH}^m,g_B)$ or $(M^n,g_0)$, if $X,Y$ are orthonormal vectors in $T_p M$, then the Riemannian sectional curvature $K(X,Y)$ at $p$ is given by O'Neill's formula (\cite{K-N2})
\begin{align*}
K(X,Y) = -\frac{c}{4}\left(1+3g(JX,Y)^2\right).
\end{align*}
For $X,Y,Z,W\in T_p M$, we also have
\begin{align*} R(X,Y,W,Z)&=R(X,Y,JW,JZ)=R(JX,JY,W,Z)\\
&=R(JX,JY,JW,JZ).
\end{align*}

We compute
\begin{align*}
R(e_1,e_2,e_2,e_1)& = K(e_1,e_2) = -\frac{c}{4}\left(1+3g(Je_1,e_2)^2\right)\\
& = -\frac{c}{4}\left(1+3g(e_2,e_2)^2\right) = -c,\\
R(e_1,e_3,e_3,e_1) & = K(e_1,e_3) = -\frac{c}{4}\left(1+3g(Je_1,e_3)^2\right)\\
& = -\frac{c}{4}\left(1+3g(e_2,e_3)^2\right) = -\frac{c}{4},\\
R(e_1,e_3,e_4,e_2) & = R(e_1,e_3,Je_4,Je_2) = R(e_1,e_3,-e_3,-e_1)\\
& = -\frac{c}{4},\\ 
R(e_1,e_4,e_3,e_2) & = R(e_1,e_4,Je_3,Je_2) = R(e_1,e_4,e_4,-e_1)\\
& = \frac{c}{4}.
\end{align*}
Using the first Bianchi identity, we compute
\begin{align*}
 R(e_1,e_2,e_4,e_3) & = -R(e_1,e_4,e_3,e_2)-R(e_1,e_3,e_2,e_4)\\
& = -R(e_1,e_4,Je_3,Je_2)-R(e_1,e_3,Je_2,Je_4)\\
& = -R(e_1,e_4,e_4,-e_1)-R(e_1,e_3,-e_1,-e_3)\\
& = R(e_1,e_4,e_4,e_1)+R(e_1,e_3,e_3,e_1)\\
& = -\frac{c}{2},
\end{align*}
Noting the patterns in the above computation, we then quickly obtain the other non-zero components. For example, $R(e_5,e_6,e_6,e_5)=-c$, $R(e_1,e_5,e_5,e_1)=R(e_1,e_5,e_6,e_2)=-R(e_1,e_6,e_5,e_2)=-c/4$, $R(e_1,e_2,e_6,e_5)=-c/2$, etc.

All the remaining components are zero, for example,
\begin{align*}
R(e_1,e_2+e_3,e_2+e_3,e_1) & = R(e_1,e_2,e_2,e_1)+R(e_1,e_3,e_3,e_1)\\
&\quad+R(e_1,e_2,e_3,e_1)+R(e_1,e_3,e_2,e_1)\\
& = -\frac{5c}{4}+2R(e_1,e_2,e_3,e_1),
\end{align*}
\begin{align*}
R(e_1,e_2+e_3,e_2+e_3,e_1)&=2R\left(e_1,\frac{e_2+e_3}{\sqrt 2},\frac{e_2+e_3}{\sqrt 2},e_1\right)\\
&=-2\frac{c}{4}\left(1+3g\left(J\left(\frac{e_2+e_3}{\sqrt 2}\right),e_1\right)^2\right)\\
&=-2\frac{c}{4}\left(1+3g\left(\frac{-e_1+e_4}{\sqrt 2},e_1\right)^2\right)\\
&=-\frac{c}{2}(1+\frac{3}{2})=-\frac{5c}{4}.\\
\end{align*}
So $R(e_1,e_2,e_3,e_1)=0$.
\end{proof}

Lemma \ref{R-gamma} is a direct consequence of Lemma \ref{entries}.
\begin{proof}[Proof of Lemma \ref{R-gamma}.]
In the $\gamma$-basis defined in Section \ref{subsect3.2} and using Lemma \ref{entries}, for example,
\begin{align*}
\left\langle R_\gamma\left(\frac{e_1e_1}{2}\right), \frac{e_1e_1}{2}\right\rangle & = R(e_1,e_1,e_1,e_1)=0,\\
\left\langle R_\gamma\left(\frac{e_1e_1}{2}\right), \frac{e_2e_2}{2}\right\rangle & = R(e_1,e_2,e_2,e_1)=-c,\\
\left\langle R_\gamma\left(\frac{e_1e_1}{2}\right), \frac{e_3e_3}{2}\right\rangle & = R(e_1,e_3,e_3,e_1)=-\frac{c}{4},\end{align*}\begin{align*}
\left\langle R_\gamma\left(\frac{e_1e_2}{\sqrt2}\right), \frac{e_1e_2}{\sqrt2}\right\rangle & = R(e_1,e_2,e_1,e_2)=c,\\
\left\langle R_\gamma\left(\frac{e_1e_2}{\sqrt 2} \right), \frac{e_3e_4}{\sqrt 2} \right\rangle & = R(e_1,e_3,e_4,e_2)+R(e_1,e_4,e_3,e_2) = 0,\\
\left\langle R_\gamma\left(\frac{e_1e_3}{\sqrt 2} \right), \frac{e_2e_4}{\sqrt 2}\right\rangle & =  R(e_1,e_2,e_4,e_3) + R(e_1,e_4,e_2,e_3) = -\frac{3c}{4},\\
\left\langle R_\gamma\left(\frac{e_1e_4}{\sqrt 2} \right), \frac{e_2e_3}{\sqrt 2} \right\rangle & = R(e_1,e_2,e_3,e_4) + R(e_1,e_3,e_2,e_4) = \frac{3c}{4}.
\end{align*}
The remaining entries of $R_\gamma$ are computed analogously.
\end{proof}

\section{Stability Theorem}\label{appxB}
We use the following version of Simonett's Stability Theorem. Please see \cite{GIK02,Kno09} for a more general version of the theorem, and \cite{Sim95} for the most general statement.
\begin{thm}[Simonett]\label{stab}
Assume the following conditions hold:
\bn
\item[(B1)] $\XX_1\overset{d}\hookrightarrow \XX_0$ and $\EE_1\overset{d}\hookrightarrow \EE_0$ are continuous and dense inclusions of Banach spaces. For fixed $0<\beta<\alpha<1$, $\XX_\alpha$ and $\XX_\beta$ are continuous interpolation spaces corresponding to the inclusion $\XX_1\overset{d}\hookrightarrow \XX_0$.
\item[(B2)] Let
\begin{align}\label{eq:quasiparabeq}
 \frac{\p}{\p t} g = Q(g)g
\end{align}
be an autonomous quasilinar parabolic equation for $t\geq 0$, with $Q(\cdot)\in C^k({\GG_\beta},\mathcal L(\XX_1,\XX_0))$ for some positive integer $k$ and some open set $\GG_\beta \subset \XX_\beta$.
\item[(B3)] For each $\hat g\in\GG_\beta$, the domain $D(Q(\hat g))$ contains $\XX_1$, and there exists an extension $\hat Q(\hat g)$ of $Q(\hat g)$ to a domain $D(\hat Q(\hat g))$ containing $\EE_1$.
\item[(B4)] For each $\hat g\in \GG_\alpha:=\GG_\beta\cap \XX_\alpha$, $\hat Q(\hat g)\in \mathcal L(\EE_1,\EE_0)$ generates a strongly continuous analytic semigroup on $\mathcal L(\EE_0,\EE_0)$.
\item[(B5)] For each $\hat g\in \GG_\alpha$, $Q(\hat g)$ agrees with the restriction of $\hat Q(\hat g)$ to the dense subset $D(Q(\hat g))\subset\XX_0$.
\item[(B6)] Let $(\EE_0,D(\hat Q(\cdot)))_\theta$ be the continuous interpolation space. Define the set $(\EE_0, D(\hat Q(\cdot)))_{1+\theta}:=\{x\in D(\hat Q(\cdot))):D(\hat Q(\cdot))(x)\in (\EE_0, D(\hat Q(\cdot))_\theta\}$ endowed with the graph norm of $\hat Q(\cdot)$ with respect to $(\EE_0,D(\hat Q(\cdot)))_\theta$. Then $\XX_0\cong (\EE_0,D(\hat Q(\cdot)))_\theta$ and $\XX_1\cong (\EE_0,D(\hat Q(\cdot)))_{1+\theta}$ for some $\theta\in(0,1)$.
\item[(B7)] $\EE_1\overset{d}\hookrightarrow \XX_\beta \overset{d}\hookrightarrow \EE_0$ with the property that there are constants $C>0$ and $\theta\in(0,1)$ such that for all $x\in\EE_1$, one has
\begin{align*}
 \Vert x\Vert_{\XX_\beta} \leq C\Vert x \Vert_{\EE_0}^{1-\theta} \Vert x \Vert_{\EE_1}^\theta.
\end{align*}
\en

For each $\alpha\in(0,1)$, let $g_0\in\GG_\alpha$ be a fixed point of equation \eqref{eq:quasiparabeq}. Suppose that the spectrum of the linearized operator $DQ|_{g_0}$ is contained in the set $\{z\in\CC:\Re(z)\leq -\varepsilon\}$ for some constant $\varepsilon >0$. Then there exist constants $\omega\in(0,\varepsilon)$ and $d_0, C_\alpha>0$, $C_\alpha$ independent of $g_0$, such that for each $d\in(0,d_0]$, one has
\begin{align*}
 \Vert \tilde g(t)-g_0 \Vert_{\XX_1} \leq \frac{C_\alpha}{t^{1-\alpha}} e^{-\omega t}\Vert \tilde g(0)-g_0\Vert_{\XX_\alpha}
\end{align*}
for all solutions $\tilde g(t)$ of equation \eqref{eq:quasiparabeq} with $\tilde g(0)\in B(\XX_\alpha,g_0,d)$, the open ball of radius $d$ centered at $g_0$ in the space $\XX_\alpha$, and for all $t\geq 0$.
\end{thm}
\bibliography{hwu_bib}
\bibliographystyle{amsplain}

\end{document}